\documentclass[12pt,twoside]{amsart}
\usepackage{latexsym,amsmath,amsopn,amssymb,amsthm,amsfonts}
\usepackage[T2A]{fontenc}
\usepackage[cp1251]{inputenc}
\usepackage[russian,english]{babel}
\usepackage{graphicx}

\newtheorem{theorem}{Theorem}

\newtheorem{proposition}{Proposition}

\newtheorem{remark}{Remark}

\newtheorem{definition}{Definition}

\newtheorem{corollary}{Corollary}

\newcommand{\Mod}{\operatorname{mod}}
\newcommand{\Spec}{\operatorname{Spec}}

{
\newcounter{table_counter}	
\newcommand{\Table}{\refstepcounter{table_counter}\newline\newline{\it Table \arabic{table_counter}}.}
\begin{document}
\begin{flushleft}
УДК 514.764.227 + 514.765 + 517.984.56 + 511
\end{flushleft}
%\begin{flushleft}
%MSC 22E30, 49J15, 53C17
%\end{flushleft}

\title[Laplace operator spectrum]{Laplace operator spectrum on connected compact simple rank four Lie groups. I}
\author{I.\,A.\,Zubareva}
\address{I.A.Zubareva}
\address{Sobolev Institute of Mathematics SD RAS, Omsk Branch, \newline 13 Pevtsova str.,  Omsk, 644099, Russia}
\email{i\_gribanova@mail.ru}
\maketitle
\maketitle {\small
\begin{quote}
\noindent{\sc Abstract.}
In this paper are given explicit calculations of Laplace operator spectrum for smooth real/complex--valued functions on all connected
compact simple rank four Lie groups with biinvariant Riemannian metric, corresponding to  root systems $B_4$, $C_4$, $D_4$ and established a connection
of obtained formulas with the number theory and integer quadratic forms from two, three and four variables.
\end{quote}}

{\small
\begin{quote}

\end{quote}}

\section*{Introduction}

In paper \cite{BerSvir1} are studied the spectrum of the Laplace(-Beltrami) operator on smooth real-valued functions defined on compact normal homogeneous
Riemannian manifolds. It was shown that in some sense this problem could be reduced to considerations of compact simply connected (connected) simple Lie groups $G$ with
biinvariant (i.e. invariant relative to left and right shifts) Riemannian metric $\nu$. In the last case it is suggested an algorithm for the search of Laplacian
spectrum via representations of Lie algebras of Lie groups $G$. In paper \cite{Svir3} this algorithm is generalized to the case
of arbitrary connected compact  simple Lie group $G$.

In our paper by means of the search algorithms for Laplacian spectrum from \cite{Svir3} we conduct explicit calculations of Laplacian spectrum for smooth
real/complex-valued
functions on all compact connected simple rank four Lie groups with biinvariant Riemannian metric, corresponding to  root systems $B_4$, $C_4$, $D_4$
and set a connection of obtained formulas to the number theory and integer quadratic forms from two, three and four variables.

The author is grateful to Professor V.N. Berestovskii for useful suggestions.

\section{Preliminaries}

Let $G$ be a compact connected simple Lie group with biinvariant Riemannian metric $\nu$.
The set $\mbox{Spec}(G,\nu)$ of all eigenvalues of Laplace--Beltrami operator $\Delta$ on smooth real-valued functions defined on $(G,\nu)$
with taking into account of multiplicity of eigenvalues, i.e. dimension of spaces of corresponding eigenfunctions, is called {\it the spectrum} of Laplace operator.
The spectrum of Lie group $(G,\nu)$ can be presented as follows:
\begin{equation}\label{Eq:spec}
\mbox{Spec}(G,\nu)=\{0=\lambda_0> \lambda_1\geq \lambda_2\geq\ldots\}.
\end{equation}
The Laplacian is naturally generalized onto complex-valued functions.

\begin{definition}
Bilinear (symmetric) form $k_{\rho}$ on a Lie algebra $\frak{g}$ defined by formula
$$k_{\rho}(u,v)={\rm trace}(\rho(u)\rho(v)),\quad u,\,v\in\frak{g},$$
is said to be the form associated with a representation $\rho$ of $\frak{g}$. The form
$k_{{\rm ad}}$, where ${\rm ad}(u)(v):=[u,v]$ is the adjoint representation of Lie algebra $\frak{g}$, is called
the Killing form of Lie algebra $\frak{g}$.
\end{definition}

We note that a compact connected Lie group $G$ is simple if and only if the adjoint representation ${\rm ad}$ of its Lie algebra $\frak{g}$ is irreducible.
In addition for any irreducible non-zero representation $\rho$ of the Lie algebra $\mathfrak{g}$ the form $k_{\rho}$ is negatively defined and is proportional
to the scalar product $\nu$.

The following proposition holds (see for example \cite{BerSvir2}).

\begin{proposition}\label{Prop:spec_ad}
If $(G,\nu)$ is a compact connected simple  $m$-dimensional Lie group with
\linebreak biinvariant Riemannian  metric $\nu$ such that $\nu(e)=-k_{{\rm ad}}$ then for the adjoint
representation ${\rm Ad}$ the Lie group $G$:
$$\lambda_{{\rm Ad}}=-1,\quad\dim{\rm Ad} = d_{{\rm Ad}} = m.$$
\end{proposition}

Ommiting details,  let us present the calculation algorithm of Laplacian spectra of all Lie groups $(G,\nu)$ with  fixed Lie algebra, stated in \cite{Svir3}
 (see corollary 5), using tables~I--IX from \cite{Burb}, where $\rho$ denotes the vector $\beta$.

\begin{theorem}
\label{alg}
To calculate Laplacian spectra for all compact connected simple Lie groups $G$ with simple Lie algebra $\frak{g}$ with root system $\Gamma$ and
biinvariant Riemannian metric $\nu$ satisfying condition $\nu(e)=-\gamma k_{ad}$, where $\gamma>0$, one needs to fulfil the following actions:

1) calculate expression
$b=\langle\tilde{\alpha}+\beta,\tilde{\alpha}+\beta\rangle-\langle\beta,\beta\rangle$,
where $\tilde{\alpha}$ is highest (maximal) root,
assuming that relative to scalar product $\langle\cdot,\cdot\rangle$  vectors $\epsilon_i$
from corresponding table in \cite{Burb} are mutually orthogonal and unitary;

2) take scalar product $\left(\cdot,\cdot\right)=\frac{1}{b}\langle\cdot,\cdot\rangle$;

3) find the fundamental weights $\bar{\omega}_1,\dots,\bar{\omega}_l$ of the Lie algebra $\frak{g}$
of the Lie group $G$ (if $\frak{g}$ has rank $l$) by corresponding table from \cite{Burb};

4) for any highest weight $\Lambda\in\Lambda^{+}_1(\Gamma)$, i.e. for any
$\Lambda=\sum\limits_{i=1}^{l}\Lambda_{i}\bar{\omega}_i$, where $\Lambda_{i}~\in~\mathbb{Z}$ and
$\Lambda_{i}\geq 0$ for $i=1,\dots,l$, find eigenvalue $\lambda(\Lambda)$ of the Laplace operator, corresponding to highest weight $\Lambda$, by formula
\begin{equation}
\label{lam}
\lambda(\Lambda)=-\frac{1}{\gamma}\Big[\langle\Lambda+\beta,\Lambda+\beta\rangle-\langle\beta,\beta\rangle\Big]
\end{equation}
and dimension $d(\Lambda+\beta)$ of irreducible complex representation of complex span of the Lie algebra $\frak{g}$ with highest weight $\Lambda$ by formula
\begin{equation}
\label{dim}
d(\Lambda+\beta)=\prod\limits_{\alpha\in\Gamma^{+}}\frac{\left(\Lambda+\beta,\alpha\right)}{{\left(\beta,\alpha\right)}};
\end{equation}

5) for any lattice $\Lambda$, satisfying relation
$\Lambda_0(\Gamma)\subseteq\Lambda\subseteq\Lambda_1(\Gamma)$, where $\Lambda_0(\Gamma)$ and
$\Lambda_1(\Gamma)$ are lattices generated by simple roots and fundamental weights, obtained in p.~1) and p.~3) respectively, $G$ is Lie group with Lie algebra $\frak{g}$,
corresponding to characteristic lattice $\Lambda=\Lambda(G)$, fulfil the following three actions:

6) find the set of highest weights $\Lambda^{+}(G)=\Lambda(G)\cap\Lambda^{+}_1(\Gamma)$, defining it via fundamental weights $\bar{\omega}_1,\dots,\bar{\omega}_l$;

7) for any highest weight $\Lambda\in\Lambda^{+}(G)$ find from p.~4) eigenvalue
$\lambda(\Lambda)$ and dimension $d(\Lambda+\beta)$ of irreducible complex representation corresponding to the weight $\Lambda$;

8) find multiplicity of any eigenvalue $\lambda=\lambda(\Lambda)$ by formula
\begin{equation}
\label{abc}
\sigma(\lambda)=\sum\limits_{\Lambda:\,\lambda(\Lambda)=\gamma\lambda}\,\prod\limits_{\;\alpha\in\Gamma^{+}}
\left(\frac{\left(\Lambda+\beta,\alpha\right)}{\left(\beta,\alpha\right)}\right)^2,
\end{equation}
obtaining in this way the spectrum $\Spec(G,\nu)$ of the Lie group $G$, corresponding to characteristic lattice $\Lambda(G)$.

Thus, we get all spectra $\Spec(G,\nu)$ of Lie groups $G$ with Lie algebra $\frak{g}$ and metric
$\nu$.
\end{theorem}

\begin{remark}
In formulas (\ref{dim}) and (\ref{abc}) applied in p.~4) and p.~8) of theorem \ref{alg} we can use instead of $\left(\cdot,\cdot\right)$ every scalar product proportional to it, in particular,
$\langle\cdot,\cdot\rangle$ from p.~1 of theorem \ref{alg}.
\end{remark}

Maximal fundamental group $\Lambda_1(\Gamma)/\Lambda_0(\Gamma)$ plays the main role in classification of simple Lie groups with given (simple) Lie algebra.
The following proposition holds (see for example \cite{BZS}).

\begin{proposition}\label{Col:LCalc}
Let maximal fundamental group $\Lambda_1(\Gamma)/\Lambda_0(\Gamma)$ of Lie algebra $\mathfrak{g}$ has prime order.
Then the family of non-isomorphic compact connected Lie groups with Lie algebra
$\mathfrak{g}$ consists of two groups: simply connected Lie group $G_1$ with the center
$\Lambda_1(\Gamma)/\Lambda_0(\Gamma)$ and weight lattice $\Lambda_1(\Gamma)$, coinciding with weight lattice of Lie algebra $\Lambda(\mathfrak{g})$,
and Lie group $G_0$ without center,
with fundamental group $\Lambda_1(\Gamma)/\Lambda_0(\Gamma)$ and weight lattice $\Lambda_0(\Gamma)$.
\end{proposition}

We see from description of all irreducible root systems in tables~I--IX from \cite{Burb} that there are only five irreducible systems of rank four:
$A_4$, $B_4$, $C_4$, $D_4$, $F_4$.

In Table \ref{Tab:list} we present the list of Lie groups, corresponding to root systems $B_4$, $C_4$, $D_4$.

\begin{table}[h]
\centering
\begin{tabular}{|p{6cm}|c|c|c|c|c|}
\hline
\centering{Lie group $G$}&$\Gamma$&$\Lambda(G)$&$\pi_1(G)$&$\mbox{C}(G)$&$\dim G$\\
\hline
\centering{$\mbox{Spin}(9)$}&$B_4$&$\Lambda_1$&$0$&$\mathbb{Z}_2$&$36$\\
\hline
\centering{$\mbox{SO}(9)$}&$B_4$&$\Lambda_0$&$\mathbb{Z}_2$&$0$&$36$\\
\hline	
\centering{$\mbox{Sp}(4)$}&$C_4$&$\Lambda_1$&$0$&$\mathbb{Z}_2$&$36$\\
\hline
\centering{$\mbox{PSp}(4)$}&$C_4$&$\Lambda_0$&$\mathbb{Z}_2$&$0$&$36$\\
\hline
\centering{$\mbox{Spin}(8)$}&$D_4$&$\Lambda_1$&$0$&$\mathbb{Z}_2\oplus\mathbb{Z}_2$&$28$\\
\hline
\centering{$\mbox{SO}(8)$}&$D_4$&$\Lambda$&$\mathbb{Z}_2$&$\mathbb{Z}_2$&$28$\\
\hline
\centering{$\mbox{PSO}(8)$}&$D_4$&$\Lambda_0$&$\mathbb{Z}_2\oplus\mathbb{Z}_2$&$0$&$28$\\
\hline
\end{tabular}
\Table\label{Tab:list} Compact connected simple rank four Lie groups, corresponding to root systems $B_4$, $C_4$, $D_4$.
\end{table}

Using theorem \ref{alg} and table \ref{Tab:list} we shall find in next sections Laplacian spectra of all compact connected simple rank four
 Lie groups corresponding to root systems $B_4$, $C_4$, $D_4$.

\section{Calculation of Laplacian spectrum for Lie groups $\mbox{Spin}(9)$ and $\mbox{SO}(9)$}

By table \ref{Tab:list}, to Lie groups under consideration corresponds the root system $B_4$. We apply table~II from \cite{Burb}.
Simple roots are
$\alpha_1=\varepsilon_1-\varepsilon_2,\,\,\alpha_2=\varepsilon_2-\varepsilon_3,\,\,\alpha_3=\varepsilon_3-\varepsilon_4,\,\,\alpha_4=\varepsilon_4$;
maximal root is $\tilde{\alpha}=\varepsilon_1+\varepsilon_2$.
Positive roots are $\varepsilon_i$, $i=1,2,3,4$; $\varepsilon_i\pm\varepsilon_j$, $1\leq i<j\leq 4$.
The sum of positive roots is equal to
$2\beta=7\varepsilon_1+5\varepsilon_2+3\varepsilon_3+\varepsilon_4$, whence
\begin{equation}
\label{s1}
\beta=\frac{7\varepsilon_1+5\varepsilon_2+3\varepsilon_3+\varepsilon_4}{2},\quad \tilde{\alpha}+\beta=
\frac{9\varepsilon_1+7\varepsilon_2+3\varepsilon_3+\varepsilon_4}{2}.
\end{equation}

We act according to algorithm suggested in theorem \ref{alg}.

1) $b=\langle\tilde{\alpha}+\beta,\tilde{\alpha}+\beta\rangle-\langle\beta,\beta\rangle=35-21=14.$

2) $\left(\cdot,\cdot\right)=\frac{1}{14}\langle\cdot,\cdot\rangle.$

3) Fundamental weights have the form
$$\bar{\omega}_1=\varepsilon_1,\quad
\bar{\omega}_2=\varepsilon_1+\varepsilon_2,\quad
\bar{\omega}_3=\varepsilon_1+\varepsilon_2+\varepsilon_3,\quad
\bar{\omega}_4=\frac{\varepsilon_1+\varepsilon_2+\varepsilon_3+\varepsilon_4}{2}.$$
It is easy to see that
$$\tilde{\alpha}=\bar{\omega}_2,\quad\beta=\bar{\omega}_1+\bar{\omega}_2+\bar{\omega}_3+\bar{\omega}_4.$$

4) Let
$\Lambda=\sum\limits_{i=1}^{4}\Lambda_i\bar{\omega}_i$, where $\Lambda_i\in\mathbb{Z}_{+}$, $i=1,2,3,4$. Then
$$\Lambda+\beta=\sum_{i=1}^{4}(\Lambda_i+1)\bar{\omega}_i=\sum_{i=1}^{4}\nu_i\bar{\omega}_i=$$
\begin{equation}
\label{s2}
=\frac{1}{2}[(2\nu_1+2\nu_2+2\nu_3+\nu_4)\varepsilon_1+
(2\nu_2+2\nu_3+\nu_4)\varepsilon_2+
(2\nu_3+\nu_4)\varepsilon_3+\nu_4\varepsilon_4],
\end{equation}
where
$\nu_i=\Lambda_i+1,\,\,\nu_i\in\mathbb{N},\,\,i=1,2,3,4.$

By formula (\ref{lam}), eigenvalue $\lambda(\Lambda),$ corresponding to highest weight
$\Lambda$, is equal to
$$\lambda(\Lambda)=-\frac{1}{14\gamma}[\langle\Lambda+\beta,\Lambda+\beta\rangle-\langle\beta,\beta\rangle]=$$
\begin{equation}
\label{b1}
=-\frac{1}{56\gamma}\left[(2\nu_1+2\nu_2+2\nu_3+\nu_4)^2+(2\nu_2+2\nu_3+\nu_4)^2+(2\nu_3+\nu_4)^2+\nu_4^2-84\right].
\end{equation}

Calculation by formula (\ref{dim}) of dimension $d(\Lambda+\beta)$ of representation $\rho(\Lambda)$
corresponding to highest weight $\Lambda$ gives
$$d(\Lambda+\beta)=\frac{(\Lambda+\beta,\varepsilon_1)}{(\beta,\varepsilon_1)}\cdot\frac{(\Lambda+\beta,\varepsilon_2)}{(\beta,\varepsilon_2)}\cdot
\frac{(\Lambda+\beta,\varepsilon_3)}{(\beta,\varepsilon_3)}\cdot\frac{(\Lambda+\beta,\varepsilon_4)}{(\beta,\varepsilon_4)}$$
$$\times\frac{(\Lambda+\beta,\varepsilon_1-\varepsilon_2)}{(\beta,\varepsilon_1-\varepsilon_2)}
\cdot\frac{(\Lambda+\beta,\varepsilon_1-\varepsilon_3)}{(\beta,\varepsilon_1-\varepsilon_3)}\cdot
\frac{(\Lambda+\beta,\varepsilon_1-\varepsilon_4)}{(\beta,\varepsilon_1-\varepsilon_4)}\cdot
\frac{(\Lambda+\beta,\varepsilon_2-\varepsilon_3)}{(\beta,\varepsilon_2-\varepsilon_3)}$$
$$\times\frac{(\Lambda+\beta,\varepsilon_2-\varepsilon_4)}{(\beta,\varepsilon_2-\varepsilon_4)}\cdot
\frac{(\Lambda+\beta,\varepsilon_3-\varepsilon_4)}{(\beta,\varepsilon_3-\varepsilon_4)}\cdot
\frac{(\Lambda+\beta,\varepsilon_1+\varepsilon_2)}{(\beta,\varepsilon_1+\varepsilon_2)}\cdot
\frac{(\Lambda+\beta,\varepsilon_1+\varepsilon_3)}{(\beta,\varepsilon_1+\varepsilon_3)}$$
$$\times\frac{(\Lambda+\beta,\varepsilon_1+\varepsilon_4)}{(\beta,\varepsilon_1+\varepsilon_4)}\cdot
\frac{(\Lambda+\beta,\varepsilon_2+\varepsilon_3)}{(\beta,\varepsilon_2+\varepsilon_3)}\cdot
\frac{(\Lambda+\beta,\varepsilon_2+\varepsilon_4)}{(\beta,\varepsilon_2+\varepsilon_4)}\cdot
\frac{(\Lambda+\beta,\varepsilon_3+\varepsilon_4)}{(\beta,\varepsilon_3+\varepsilon_4)}.$$
On the ground of (\ref{s1}) and (\ref{s2}),
$$\frac{(\Lambda+\beta,\varepsilon_1)}{(\beta,\varepsilon_1)}=
\frac{2\nu_1+2\nu_2+2\nu_3+\nu_4}{7},\quad
\frac{(\Lambda+\beta,\varepsilon_2)}{(\beta,\varepsilon_2)}=
\frac{2\nu_2+2\nu_3+\nu_4}{5},$$
$$\frac{(\Lambda+\beta,\varepsilon_3)}{(\beta,\varepsilon_3)}=
\frac{2\nu_3+\nu_4}{3},\quad
\frac{(\Lambda+\beta,\varepsilon_4)}{(\beta,\varepsilon_4)}=
\nu_4,\quad
\frac{(\Lambda+\beta,\varepsilon_1-\varepsilon_2)}{(\beta,\varepsilon_1-\varepsilon_2)}=
\nu_1,$$
$$\frac{(\Lambda+\beta,\varepsilon_1-\varepsilon_3)}{(\beta,\varepsilon_1-\varepsilon_3)}=
\frac{\nu_1+\nu_2}{2},\quad
\frac{(\Lambda+\beta,\varepsilon_1-\varepsilon_4)}{(\beta,\varepsilon_1-\varepsilon_4)}=
\frac{\nu_1+\nu_2+\nu_3}{3},$$
$$\frac{(\Lambda+\beta,\varepsilon_2-\varepsilon_3)}{(\beta,\varepsilon_2-\varepsilon_3)}=
\nu_2,\quad
\frac{(\Lambda+\beta,\varepsilon_2-\varepsilon_4)}{(\beta,\varepsilon_2-\varepsilon_4)}=
\frac{\nu_2+\nu_3}{2},\quad
\frac{(\Lambda+\beta,\varepsilon_3-\varepsilon_4)}{(\beta,\varepsilon_3-\varepsilon_4)}=
\nu_3,$$
$$\frac{(\Lambda+\beta,\varepsilon_1+\varepsilon_2)}{(\beta,\varepsilon_1+\varepsilon_2)}=
\frac{\nu_1+2\nu_2+2\nu_3+\nu_4}{6},\quad
\frac{(\Lambda+\beta,\varepsilon_1+\varepsilon_3)}{(\beta,\varepsilon_1+\varepsilon_3)}=
\frac{\nu_1+\nu_2+2\nu_3+\nu_4}{5},$$
$$\frac{(\Lambda+\beta,\varepsilon_1+\varepsilon_4)}{(\beta,\varepsilon_1+\varepsilon_4)}=
\frac{\nu_1+\nu_2+\nu_3+\nu_4}{4},\quad
\frac{(\Lambda+\beta,\varepsilon_2+\varepsilon_3)}{(\beta,\varepsilon_2+\varepsilon_3)}=
\frac{\nu_2+2\nu_3+\nu_4}{4},$$
$$\frac{(\Lambda+\beta,\varepsilon_2+\varepsilon_4)}{(\beta,\varepsilon_2+\varepsilon_4)}=
\frac{\nu_2+\nu_3+\nu_4}{3},\quad
\frac{(\Lambda+\beta,\varepsilon_3+\varepsilon_4)}{(\beta,\varepsilon_3+\varepsilon_4)}=
\frac{\nu_3+\nu_4}{2}.$$
Therefore
\begin{equation}
\label{b2}
d(\Lambda+\beta)=\frac{1}{3628800}\nu_1\nu_2\nu_3\nu_4(\nu_1+\nu_2)(\nu_2+\nu_3)(\nu_3+\nu_4)(2\nu_3+\nu_4)(\nu_1+\nu_2+\nu_3)
\end{equation}
$$\times(\nu_2+\nu_3+\nu_4)(\nu_2+2\nu_3+\nu_4)(2\nu_2+2\nu_3+\nu_4)
(\nu_1+\nu_2+\nu_3+\nu_4)$$
$$\times(\nu_1+2\nu_2+2\nu_3+\nu_4)(\nu_1+\nu_2+2\nu_3+\nu_4)(2\nu_1+2\nu_2+2\nu_3+\nu_4).$$

5) Simple roots and fundamental weights indicated above generate respective lattices
$\Lambda_0(B_4)$ and $\Lambda_1(B_4)$, i.e.
\begin{equation}
\label{b3}
\Lambda_0(B_4)=\left\{\sum\limits_{i=1}^{4}\Psi_i\alpha_i\,\mid\,\Psi_i\in\mathbb{Z},\,\,i=1,2,3,4\right\},
\end{equation}
\begin{equation}
\label{b30}
\Lambda_1(B_4)=\left\{\sum\limits_{i=1}^{4}\Lambda_i\bar{\omega}_i\,\mid\,\Lambda_i\in\mathbb{Z},\,\,i=1,2,3,4\right\}.
\end{equation}

After expressing roots via fundamental weights
$$\alpha_1=2\bar{\omega}_1-\bar{\omega}_2,\,\,
\alpha_2=-\bar{\omega}_1+2\bar{\omega}_2-\bar{\omega}_3,\,\,
\alpha_3=-\bar{\omega}_2+2\bar{\omega}_3-2\bar{\omega}_4,\,\,
\alpha_4=-\bar{\omega}_3+2\bar{\omega}_4$$
and the change of variables
$$\left\{
\begin{array}{l}
\Psi_1=\Omega_1+\Omega_2+\Omega_3+\Omega_4, \\
\Psi_2=\Omega_1+2\Omega_2+2\Omega_3+2\Omega_4, \\
\Psi_3=\Omega_1+2\Omega_2+3\Omega_3+3\Omega_4, \\
\Psi_4=\Omega_1+2\Omega_2+3\Omega_3+4\Omega_4
\end{array}\right.\quad\Leftrightarrow\quad
\left\{
\begin{array}{l}
\Omega_1=2\Psi_1-\Psi_2, \\
\Omega_2=-\Psi_1+2\Psi_2-\Psi_3, \\
\Omega_3=-\Psi_2+2\Psi_3-\Psi_4, \\
\Omega_4=-\Psi_3+\Psi_4
\end{array}\right.$$
the lattice $\Lambda_0(B_4)$ takes the following view
\begin{equation}
\label{b4}
\Lambda_0(B_4)=\{\Omega_1\bar{\omega}_1+\Omega_2\bar{\omega}_2+\Omega_3\bar{\omega}_3+2\Omega_4\bar{\omega}_4\,\mid\,\Omega_i\in\mathbb{Z},\,\,i=1,2,3,4\}.
\end{equation}

It follows from (\ref{b3}), (\ref{b30}), and (\ref{b4}) that $\Lambda_0(B_4)\subset\Lambda_1(B_4)$ and
$\Lambda_1(B_4)/\Lambda_0(B_4)\cong\mathbb{Z}_2$, i.e. it has prime order. Hence on the ground of proposition \ref{Col:LCalc} there is no other lattices.
Lie groups, associated with these characteristic lattices, are presented in table \ref{Tab:list}.

a) Let us give formulas defining Laplacian spectrum of the Lie group  $\mbox{Spin}(9)$.

6a) By formula (\ref{b3}), the set of highest weights of the Lie group  $\mbox{Spin}(9)$ is equal to
$$\Lambda^{+}(\mbox{Spin}(9))=\left\{\sum\limits_{i=1}^{4}\Lambda_i\bar{\omega}_i\,\mid\,\Lambda_i\in\mathbb{Z}_{+},\,\,i=1,2,3,4\right\}.$$

7a) Set $\Lambda=\sum\limits_{i=1}^{4}\Lambda_i\bar{\omega}_i\in\Lambda^{+}(\mbox{Spin}(9))$,
then by p.~4) eigenvalue $\lambda(\Lambda)$ and dimension $d(\Lambda+\beta)$
are calculated by formulas   (\ref{b1}) and (\ref{b2}) respectively, where $\nu_i=\Lambda_i+1$, $\nu_i\in\mathbb{N}$, $i=1,2,3,4.$

8a) Applying formula (\ref{abc}) and results of preceding point, we get the multiplicity of eigenvalue $\lambda(\Lambda)$:
$$\sigma(\Lambda)=\frac{1}{3628800^2}\sum_{\Xi_4}
\nu_1^2\nu_2^2\nu_3^2\nu_4^2(\nu_1+\nu_2)^2(\nu_2+\nu_3)^2(\nu_3+\nu_4)^2(2\nu_3+\nu_4)^2$$
$$\times(\nu_1+\nu_2+\nu_3)^2(\nu_2+\nu_3+\nu_4)^2(\nu_2+2\nu_3+\nu_4)^2(2\nu_2+2\nu_3+\nu_4)^2$$
$$\times(\nu_1+\nu_2+\nu_3+\nu_4)^2(\nu_1+2\nu_2+2\nu_3+\nu_4)^2(\nu_1+\nu_2+2\nu_3+\nu_4)^2(2\nu_1+2\nu_2+2\nu_3+\nu_4)^2,$$
where
\begin{equation}
\label{xi4}
\begin{array}{r l l}
&&\Xi_1=\{(2\nu_1+2\nu_2+2\nu_3+\nu_4)^2+(2\nu_2+2\nu_3+\nu_4)^2+(2\nu_3+\nu_4)^2+\\
&&+\nu_4^2=84-56\gamma\lambda\,\mid\,\nu_i\in\mathbb{N},\,i=1,2,3,4\}.
\end{array}
\end{equation}

The least by modulus non-zero eigenvalue of Laplacian is equal to $-\frac{4}{7\gamma}$ and corresponds to irreducible complex representation of the Lie group
 $\mbox{Spin}(9)$ with highest weight
$\bar{\omega}_1$. The dimension of this representation is equal to $9$. Therefore the multiplicity of the eigenvalue $-\frac{4}{7\gamma}$ is equal to $9^2=81$.

b) Let us give formulas defining Laplacian spectrum of the Lie group $\mbox{SO}(9)$.

6b) By formula (\ref{b4}),  the set of highest weights of the Lie group  $\mbox{SO}(9)$ is equal to
$$\Lambda^{+}(\mbox{SO}(9))=\{\Omega_1\bar{\omega}_1+\Omega_2\bar{\omega}_2+\Omega_3\bar{\omega}_3+2\Omega_4\bar{\omega}_4\,\,\mid\,
\Omega_i\in\mathbb{Z}_{+},\,\,i=1,2,3,4\}.$$

7b) Set $\Lambda=\Omega_1\bar{\omega}_1+\Omega_2\bar{\omega}_2+\Omega_3\bar{\omega}_3+2\Omega_4\bar{\omega}_4\in\Lambda^{+}(\mbox{SO}(9))$,
then by p.~4) eigenvalue $\lambda(\Lambda)$ and dimension $d(\Lambda+\beta)$ are calculated by formulas (\ref{b1}) and (\ref{b2}) respectively, where
$\nu_1=\Omega_1+1$, $\nu_2=\Omega_2+1$, $\nu_3=\Omega_3+1$, $\nu_4=2\Omega_4+1$, $\nu_i\in\mathbb{N}$, $i=1,2,3,4,$ $\nu_4\equiv 1({\rm mod}\,2)$.

8b) Applying formula (\ref{abc}) and results of preceding point, we get the following multiplicity of the eigenvalue $\lambda(\Lambda)$:
$$\sigma(\Lambda)=\frac{1}{3628800^2}\sum_{\Xi_5}
\nu_1^2\nu_2^2\nu_3^2\nu_4^2(\nu_1+\nu_2)^2(\nu_2+\nu_3)^2(\nu_3+\nu_4)^2(2\nu_3+\nu_4)^2$$
$$\times(\nu_1+\nu_2+\nu_3)^2(\nu_2+\nu_3+\nu_4)^2(\nu_2+2\nu_3+\nu_4)^2(2\nu_2+2\nu_3+\nu_4)^2$$
$$\times(\nu_1+\nu_2+\nu_3+\nu_4)^2(\nu_1+2\nu_2+2\nu_3+\nu_4)^2(\nu_1+\nu_2+2\nu_3+\nu_4)^2(2\nu_1+2\nu_2+2\nu_3+\nu_4)^2,$$
where
\begin{equation}
\label{xi5}
\begin{array}{r c l}
&&\Xi_2=\{(2\nu_1+2\nu_2+2\nu_3+\nu_4)^2+(2\nu_2+2\nu_3+\nu_4)^2+(2\nu_3+\nu_4)^2+\\
&&+\nu_4^2=84-56\gamma\lambda\,\mid\,\nu_4\equiv 1({\rm mod}\,2),\,\,\nu_i\in\mathbb{N},\,\,i=1,2,3,4\}.
\end{array}
\end{equation}

The least by modulus non-zero eigenvalue of Laplacian is equal to $-\frac{4}{7\gamma}$ and corresponds to irreducible complex representation of the Lie group
$\mbox{SO}(9)$  with highest weight
$\bar{\omega}_1$. The dimension of this representation is equal to $9$. Consequently the multiplicity of the eigenvalue $-\frac{4}{7\gamma}$ is equal to $9^2=81$.

\section{Calculation of Laplacian spectrum for Lie groups $\mbox{Sp}(4)$ and $\mbox{Sp}(4)/\mbox{C}(\mbox{Sp}(4))$}

By table \ref{Tab:list}, the Lie groups under considerations correspond to root system $C_4$.
We apply table~III from \cite{Burb}.
Simple roots are
$\alpha_1=\varepsilon_1-\varepsilon_2,\,\,\alpha_2=\varepsilon_2-\varepsilon_3,\,\,\alpha_3=\varepsilon_3-\varepsilon_4,\,\,\alpha_4=2\varepsilon_4$;
maximal root is $\tilde{\alpha}=2\varepsilon_1$.
Positive roots are $2\varepsilon_i$, $i=1,2,3,4$; $\varepsilon_i\pm\varepsilon_j$, $1\leq i<j\leq 4$.
The sum of positive roots is equal to
$2\beta=8\varepsilon_1+6\varepsilon_2+4\varepsilon_3+2\varepsilon_4$, whence
\begin{equation}
\label{ss1}
\beta=4\varepsilon_1+3\varepsilon_2+2\varepsilon_3+\varepsilon_4,\quad \tilde{\alpha}+\beta=6\varepsilon_1+3\varepsilon_2+2\varepsilon_3+\varepsilon_4.
\end{equation}

We act according to algorithm presented in theorem \ref{alg}.

1) $b=\langle\tilde{\alpha}+\beta,\tilde{\alpha}+\beta\rangle-\langle\beta,\beta\rangle=50-30=20.$

2) $\left(\cdot,\cdot\right)=\frac{1}{20}\langle\cdot,\cdot\rangle.$

3) Fundamental weights have the following form
$$\bar{\omega}_1=\varepsilon_1,\quad
\bar{\omega}_2=\varepsilon_1+\varepsilon_2,\quad
\bar{\omega}_3=\varepsilon_1+\varepsilon_2+\varepsilon_3,\quad
\bar{\omega}_4=\varepsilon_1+\varepsilon_2+\varepsilon_3+\varepsilon_4.$$
It is easy to see that
$$\tilde{\alpha}=2\bar{\omega}_1,\quad\beta=\bar{\omega}_1+\bar{\omega}_2+\bar{\omega}_3+\bar{\omega}_4.$$

4) Set
$\Lambda=\sum\limits_{i=1}^{4}\Lambda_i\bar{\omega}_i$, where $\Lambda_i\in\mathbb{Z}_{+}$, $i=1,2,3,4$. Then
\begin{equation}
\label{ss2}
\Lambda+\beta=\sum_{i=1}^{4}(\Lambda_i+1)\bar{\omega}_i=
(\nu_1+\nu_2+\nu_3+\nu_4)\varepsilon_1+
(\nu_2+\nu_3+\nu_4)\varepsilon_2+
(\nu_3+\nu_4)\varepsilon_3+\nu_4\varepsilon_4,
\end{equation}
where
$\nu_i=\Lambda_i+1,\,\,\nu_i\in\mathbb{N},\,\,i=1,2,3,4.$

By formula (\ref{lam}), eigenvalue $\lambda(\Lambda)$, corresponding to highest weight
$\Lambda$, is equal to
$$\lambda(\Lambda)=-\frac{1}{20\gamma}[\langle\Lambda+\beta,\Lambda+\beta\rangle-\langle\beta,\beta\rangle]=$$
\begin{equation}
\label{c1}
=-\frac{1}{20\gamma}\left[(\nu_1+\nu_2+\nu_3+\nu_4)^2+(\nu_2+\nu_3+\nu_4)^2+(\nu_3+\nu_4)^2+\nu_4^2-30\right].
\end{equation}

By formula (\ref{dim}), dimension $d(\Lambda+\beta)$ of representation
$\rho(\Lambda)$, associated with highest weight $\Lambda$, is equal to
$$d(\Lambda+\beta)=\frac{(\Lambda+\beta,2\varepsilon_1)}{(\beta,2\varepsilon_1)}\cdot\frac{(\Lambda+\beta,2\varepsilon_2)}{(\beta,2\varepsilon_2)}\cdot
\frac{(\Lambda+\beta,2\varepsilon_3)}{(\beta,2\varepsilon_3)}
\cdot\frac{(\Lambda+\beta,2\varepsilon_4)}{(\beta,2\varepsilon_4)}$$
$$\times\frac{(\Lambda+\beta,\varepsilon_1-\varepsilon_2)}{(\beta,\varepsilon_1-\varepsilon_2)}
\cdot\frac{(\Lambda+\beta,\varepsilon_1-\varepsilon_3)}{(\beta,\varepsilon_1-\varepsilon_3)}\cdot
\frac{(\Lambda+\beta,\varepsilon_1-\varepsilon_4)}{(\beta,\varepsilon_1-\varepsilon_4)}\cdot
\frac{(\Lambda+\beta,\varepsilon_2-\varepsilon_3)}{(\beta,\varepsilon_2-\varepsilon_3)}$$
$$\times\frac{(\Lambda+\beta,\varepsilon_2-\varepsilon_4)}{(\beta,\varepsilon_2-\varepsilon_4)}\cdot
\frac{(\Lambda+\beta,\varepsilon_3-\varepsilon_4)}{(\beta,\varepsilon_3-\varepsilon_4)}\cdot
\frac{(\Lambda+\beta,\varepsilon_1+\varepsilon_2)}{(\beta,\varepsilon_1+\varepsilon_2)}\cdot
\frac{(\Lambda+\beta,\varepsilon_1+\varepsilon_3)}{(\beta,\varepsilon_1+\varepsilon_3)}$$
$$\times\frac{(\Lambda+\beta,\varepsilon_1+\epsilon_4)}{(\beta,\varepsilon_1+\varepsilon_4)}\cdot
\frac{(\Lambda+\beta,\varepsilon_2+\varepsilon_3)}{(\beta,\varepsilon_2+\varepsilon_3)}\cdot
\frac{(\Lambda+\beta,\varepsilon_2+\varepsilon_4)}{(\beta,\varepsilon_2+\varepsilon_4)}\cdot
\frac{(\Lambda+\beta,\varepsilon_3+\varepsilon_4)}{(\beta,\varepsilon_3+\varepsilon_4)}.$$
On the ground of  (\ref{ss1}) and (\ref{ss2}),
$$\frac{(\Lambda+\beta,2\varepsilon_1)}{(\beta,2\varepsilon_1)}=
\frac{\nu_1+\nu_2+\nu_3+\nu_4}{4},\quad
\frac{(\Lambda+\beta,2\varepsilon_2)}{(\beta,2\varepsilon_2)}=
\frac{\nu_2+\nu_3+\nu_4}{3},$$
$$\frac{(\Lambda+\beta,2\varepsilon_3)}{(\beta,2\varepsilon_3)}=
\frac{\nu_3+\nu_4}{2},\quad
\frac{(\Lambda+\beta,2\varepsilon_4)}{(\beta,2\varepsilon_4)}=
\nu_4,\quad
\frac{(\Lambda+\beta,\varepsilon_1-\varepsilon_2)}{(\beta,\varepsilon_1-\varepsilon_2)}=
\nu_1,$$
$$\frac{(\Lambda+\beta,\varepsilon_1-\varepsilon_3)}{(\beta,\varepsilon_1-\varepsilon_3)}=
\frac{\nu_1+\nu_2}{2},\quad
\frac{(\Lambda+\beta,\varepsilon_1-\varepsilon_4)}{(\beta,\varepsilon_1-\varepsilon_4)}=
\frac{\nu_1+\nu_2+\nu_3}{3},$$
$$\frac{(\Lambda+\beta,\varepsilon_2-\varepsilon_3)}{(\beta,\varepsilon_2-\varepsilon_3)}=
\nu_2,\quad
\frac{(\Lambda+\beta,\varepsilon_2-\varepsilon_4)}{(\beta,\varepsilon_2-\varepsilon_4)}=
\frac{\nu_2+\nu_3}{2},\quad
\frac{(\Lambda+\beta,\varepsilon_3-\varepsilon_4)}{(\beta,\varepsilon_3-\varepsilon_4)}=
\nu_3,$$
$$\frac{(\Lambda+\beta,\varepsilon_1+\varepsilon_2)}{(\beta,\varepsilon_1+\varepsilon_2)}=
\frac{\nu_1+2\nu_2+2\nu_3+2\nu_4}{7},\quad
\frac{(\Lambda+\beta,\varepsilon_1+\varepsilon_3)}{(\beta,\varepsilon_1+\varepsilon_3)}=
\frac{\nu_1+\nu_2+2\nu_3+2\nu_4}{6},$$
$$\frac{(\Lambda+\beta,\varepsilon_1+\varepsilon_4)}{(\beta,\varepsilon_1+\varepsilon_4)}=
\frac{\nu_1+\nu_2+\nu_3+2\nu_4}{5},\quad
\frac{(\Lambda+\beta,\varepsilon_2+\varepsilon_3)}{(\beta,\varepsilon_2+\varepsilon_3)}=
\frac{\nu_2+2\nu_3+2\nu_4}{5},$$
$$\frac{(\Lambda+\beta,\varepsilon_2+\varepsilon_4)}{(\beta,\varepsilon_2+\varepsilon_4)}=
\frac{\nu_2+\nu_3+2\nu_4}{4},\quad
\frac{(\Lambda+\beta,\varepsilon_3+\varepsilon_4)}{(\beta,\varepsilon_3+\varepsilon_4)}=
\frac{\nu_3+2\nu_4}{3}.$$
Consequently
\begin{equation}
\label{c2}
d(\Lambda+\beta)=\frac{1}{3628800}\nu_1\nu_2\nu_3\nu_4(\nu_1+\nu_2)(\nu_2+\nu_3)(\nu_3+\nu_4)(\nu_3+2\nu_4)
\end{equation}
$$\times(\nu_1+\nu_2+\nu_3)(\nu_2+\nu_3+\nu_4)(\nu_2+\nu_3+2\nu_4)(\nu_2+2\nu_3+2\nu_4)
(\nu_1+\nu_2+\nu_3+\nu_4)$$
$$\times(\nu_1+\nu_2+2\nu_3+2\nu_4)(\nu_1+\nu_2+\nu_3+2\nu_4)(\nu_1+2\nu_2+2\nu_3+2\nu_4).$$

5) Simple roots and fundamental weights of the Lie algebra $\mathfrak{sp}(4)$ indicated above
define respective lattices $\Lambda_0(C_4)$ and $\Lambda_1(C_4)$, i.e.\begin{equation}
\label{c3}
\Lambda_0(C_4)=\left\{\sum\limits_{i=1}^{4}\Psi_i\alpha_i\,\mid\,\Psi_i\in\mathbb{Z},\,\,i=1,2,3,4\right\},
\end{equation}
$$\Lambda_1(C_4)=\left\{\sum\limits_{i=1}^{4}\Lambda_i\bar{\omega}_i\,\mid\,\Lambda_i\in\mathbb{Z},\,\,i=1,2,3,4\right\}.$$

After expressing roots via fundamental weights
$$\alpha_1=2\bar{\omega}_1-\bar{\omega}_2,\,\,\alpha_2=-\bar{\omega}_1+2\bar{\omega}_2-\bar{\omega}_3,\,\,
\alpha_3=-\bar{\omega}_2+2\bar{\omega}_3-\bar{\omega}_4,\,\,\alpha_4=-2\bar{\omega}_3+2\bar{\omega}_4$$
and the change of variables
$$\left\{
\begin{array}{l}
\Psi_1=-2\Omega_1+\Omega_2+2\Omega_4, \\
\Psi_2=2\Omega_1+2\Omega_2+3\Omega_3+2\Omega_4, \\
\Psi_3=-2\Omega_1+2\Omega_2+2\Omega_3+3\Omega_4, \\
\Psi_4=-\Omega_1+\Omega_2+\Omega_3+2\Omega_4
\end{array}\right.\quad\Leftrightarrow\quad
\left\{
\begin{array}{l}
\Omega_1=\Psi_1-\Psi_3+\Psi_4, \\
\Omega_2=-\Psi_1+2\Psi_2-\Psi_3, \\
\Omega_3=-\Psi_2+2\Psi_3-2\Psi_4, \\
\Omega_4=-\Psi_3+2\Psi_4
\end{array}\right.$$
the lattice $\Lambda_0(C_4)$ takes the following form
\begin{equation}
\label{c4}
\Lambda_0(C_4)=\{2\Omega_1\bar{\omega}_1+\Omega_2\bar{\omega}_2+\Omega_3(\bar{\omega}_1+\bar{\omega}_3)+\Omega_4\bar{\omega}_4\,\,\mid\,
\Omega_i\in\mathbb{Z},\,\,i=1,2,3,4\}.
\end{equation}
Let us define lattice  $\Lambda_1(C_4)$ in base $\{\bar{\omega}_1,\bar{\omega}_2,\bar{\omega}_1+\bar{\omega}_3,\bar{\omega}_4\}$ via the change of variables
$\{\Lambda_1=\Omega_1+\Omega_3,\,\Lambda_2=\Omega_2,\,\Lambda_3=\Omega_3,\,\Lambda_4=\Omega_4\}$:
\begin{equation}
\label{c5}
\Lambda_1(C_4)=\{\Omega_1\bar{\omega}_1+\Omega_2\bar{\omega}_2+\Omega_3(\bar{\omega}_1+\bar{\omega}_3)+\Omega_4\bar{\omega}_4\,\,\mid\,
\Omega_i\in\mathbb{Z},\,\,i=1,2,3,4\}.
\end{equation}

It follows from (\ref{c4}) and (\ref{c5}) that $\Lambda_0(C_4)\subset\Lambda_1(C_4)$ and
$\Lambda_1(C_4)/\Lambda_0(C_4)\cong\mathbb{Z}_2$, i.e. it has prime order. It follows from here and proposition \ref{Col:LCalc}
that there is no other lattices. Lie groups, associated with these characteristic lattices, are presented in table  \ref{Tab:list}.

a) Let us give formulas, defining Laplacian spectrum of the Lie group ${\rm Sp}(4)$.

6a) By formula (\ref{c3}), the set of highest weights of the Lie group  $\mbox{Sp}(4)$ is equal to
$$\Lambda^{+}(\mbox{Sp}(4))=\left\{\sum\limits_{i=1}^{4}\Lambda_i\bar{\omega}_i\,\mid\,\Lambda_i\in\mathbb{Z}_{+},\,\,i=1,2,3,4\right\}.$$

7a) Set $\Lambda=\sum\limits_{i=1}^{4}\Lambda_i\bar{\omega}_i\in\Lambda^{+}({\rm Sp}(4))$,
then by p.~4) eigenvalue  $\lambda(\Lambda)$ and dimension $d(\Lambda+\beta)$
are calculated respectively by formulas (\ref{c1}) and (\ref{c2}), where $\nu_i=\Lambda_i+1$, $\nu_i\in\mathbb{N}$, $i=1,2,3,4.$

8a) Applying formula (\ref{abc}) and result of preceding point, we get the following multiplicity of eigenvalue $\lambda$
$$\sigma(\Lambda)=\frac{1}{3628800^2}\sum_{\Xi_6}\nu_1^2\nu_2^2\nu_3^2\nu_4^2(\nu_1+\nu_2)^2(\nu_2+\nu_3)^2(\nu_3+\nu_4)^2(\nu_3+2\nu_4)^2$$
$$\times(\nu_1+\nu_2+\nu_3)^2(\nu_2+\nu_3+\nu_4)^2(\nu_2+\nu_3+2\nu_4)^2(\nu_2+2\nu_3+2\nu_4)^2$$
$$\times(\nu_1+\nu_2+\nu_3+\nu_4)^2(\nu_1+\nu_2+2\nu_3+2\nu_4)^2(\nu_1+\nu_2+\nu_3+2\nu_4)^2(\nu_1+2\nu_2+2\nu_3+2\nu_4)^2,$$
where
\begin{equation}
\label{xi6}
\begin{array}{r c l}
&&\Xi_3=\{(\nu_1+\nu_2+\nu_3+\nu_4)^2+(\nu_2+\nu_3+\nu_4)^2+(\nu_3+\nu_4)^2+\\
&&+\nu_4^2=30-20\gamma\lambda\,\mid\,\nu_i\in\mathbb{N},\,\,i=1,2,3,4\}.
\end{array}
\end{equation}

The least by modulus non-zero eigenvalue of Laplacian is equal to $-\frac{9}{20\gamma}$ and corresponds to irreducible complex representation
of the Lie group ${\rm Sp}(4)$
with highest weight $\bar{\omega}_1$. The dimension of this representation is equal to $8$. Therefore the multiplicity of eigenvalue
$-\frac{9}{20\gamma}$ is equal to $8^2=64$.

b) Let present formulas giving Laplacian spectrum of the Lie group
 $\mbox{Sp}(4)/\mbox{C}(\mbox{Sp}(4))$.

6b) By formula (\ref{c4}), the  set of highest weights of the Lie group $\mbox{Sp}(4)/\mbox{C}(\mbox{Sp}(4))$ is equal to
$$\begin{array}{r c l}
&&\Lambda^{+}(\mbox{Sp}(4)/\mbox{C}(\mbox{Sp}(4)))=
\{(2\Omega_1+\Omega_3)\bar{\omega}_1+\Omega_2\bar{\omega}_2+\Omega_3\bar{\omega}_3+\Omega_4\bar{\omega}_4\mid\\
&&\Omega_i\in\mathbb{Z},\,i=1,2,3,4;\,2\Omega_1+\Omega_3\geq 0,\,\Omega_2\geq 0,\,\Omega_3\geq 0,\,\Omega_4\geq 0\}.
\end{array}$$

7b) Set $\Lambda=(2\Omega_1+\Omega_3)\bar{\omega}_1+\Omega_2\bar{\omega}_2+\Omega_3\bar{\omega}_3+\Omega_4
\bar{\omega}_4\in\Lambda^{+}(\mbox{Sp}(4)/\mbox{C}(\mbox{Sp}(4)))$, then by p.~4) eigenvalue
$\lambda(\Lambda)$ and dimension
$d(\Lambda+\beta)$ are calculated respectively by formulas (\ref{c1}) and (\ref{c2}), where
$\nu_1=2\Omega_1+\Omega_3+1$, $\nu_2=\Omega_2+1$, $\nu_3=\Omega_3+1$, $\nu_4=\Omega_4+1$, $\nu_i\in\mathbb{N}$, $i=1,2,3,4;$ $\nu_1\equiv \nu_3(\mbox{mod}\,2)$.

8b) Applying formula (\ref{abc}) and results of preceding point, we get the following multiplicity of eigenvalue $\lambda(\Lambda)$
$$\sigma(\Lambda)=\frac{1}{3628800^2}\sum_{\Xi_7}\nu_1^2\nu_2^2\nu_3^2\nu_4^2(\nu_1+\nu_2)^2(\nu_2+\nu_3)^2(\nu_3+\nu_4)^2(\nu_3+2\nu_4)^2$$
$$\times(\nu_1+\nu_2+\nu_3)^2(\nu_2+\nu_3+\nu_4)^2(\nu_2+\nu_3+2\nu_4)^2(\nu_2+2\nu_3+2\nu_4)^2$$
$$\times(\nu_1+\nu_2+\nu_3+\nu_4)^2(\nu_1+\nu_2+2\nu_3+2\nu_4)^2(\nu_1+\nu_2+\nu_3+2\nu_4)^2(\nu_1+2\nu_2+2\nu_3+2\nu_4)^2,$$
where
\begin{equation}
\label{xi7}
\begin{array}{r c l}
&&\Xi_4=\{(\nu_1+\nu_2+\nu_3+\nu_4)^2+(\nu_2+\nu_3+\nu_4)^2+(\nu_3+\nu_4)^2+\\
&&+\nu_4^2=30-20\gamma\lambda\,\mid\,\nu_1\equiv\nu_3(\mbox{mod}\,2),\,\,\nu_i\in\mathbb{N},\,\,i=1,2,3,4\}.
\end{array}
\end{equation}

The least by modulus non-zero eigenvalue of Laplacian is equal to $-\frac{4}{5\gamma}$ and corresponds to irreducible complex representation
 of the Lie group $\mbox{Sp}(4)/\mbox{C}(\mbox{Sp}(4))$  with highest weight $\bar{\omega}_2$. The dimension of this representation is equal to $27$.
 Consequently the multiplicity of the eigenvalue $-\frac{4}{5\gamma}$ is equal to $27^2=729$.

\section{Calculation of Laplacian spectrum for Lie groups $\mbox{Spin}(8)$, $\mbox{SO}(8)$, and $\mbox{PSO}(8)$}

By table  \ref{Tab:list}, the Lie groups under considerations correspond to root system $D_4$.
We apply table~IV from \cite{Burb}.
Simple roots are $\alpha_1=\varepsilon_1-\varepsilon_2,\,\,\alpha_2=\varepsilon_2-\varepsilon_3,\,\,\alpha_3=\varepsilon_3-\varepsilon_4,\,\,
\alpha_4=\varepsilon_3+\varepsilon_4$; maximal root is $\tilde{\alpha}=\varepsilon_1+\varepsilon_2$.
Positive roots are $\varepsilon_i\pm\varepsilon_j$, $1\leq i<j\leq 4$.
The sum of positive roots is equal to
$2\beta=6\varepsilon_1+4\varepsilon_2+2\varepsilon_3$, whence
\begin{equation}
\label{m1}
\beta=3\varepsilon_1+2\varepsilon_2+\varepsilon_3,\quad \tilde{\alpha}+\beta=4\varepsilon_1+3\varepsilon_2+\varepsilon_3.
\end{equation}

We act according to algorithm presented in theorem \ref{alg}.

1) $b=\langle\tilde{\alpha}+\beta,\tilde{\alpha}+\beta\rangle-\langle\beta,\beta\rangle=26-14=12.$

2) $\left(\cdot,\cdot\right)=\frac{1}{12}\langle\cdot,\cdot\rangle.$

3) Fundamental weights have the following form
$$\bar{\omega}_1=\varepsilon_1,\quad
\bar{\omega}_2=\varepsilon_1+\varepsilon_2,\quad
\bar{\omega}_3=\frac{\varepsilon_1+\varepsilon_2+\varepsilon_3-\varepsilon_4}{2},\quad
\bar{\omega}_4=\frac{\varepsilon_1+\varepsilon_2+\varepsilon_3+\varepsilon_4}{2}.$$
It is easy to see that
$$\tilde{\alpha}=\bar{\omega}_2,\quad\beta=\bar{\omega}_1+\bar{\omega}_2+\bar{\omega}_3+\bar{\omega}_4.$$

4) Set
$\Lambda=\sum\limits_{i=1}^{4}\Lambda_i\bar{\omega}_i$, where $\Lambda_i\in\mathbb{Z}_{+}$, $i=1,2,3,4$. Then
\begin{equation}
\label{m2}
\Lambda+\beta=\sum_{i=1}^{4}(\Lambda_i+1)\bar{\omega}_i=
\sum_{i=1}^{4}\nu_i\bar{\omega}_i=
\end{equation}
$$=\frac{1}{2}
[(2\nu_1+2\nu_2+\nu_3+\nu_4)\varepsilon_1+
(2\nu_2+\nu_3+\nu_4)\varepsilon_2+
(\nu_3+\nu_4)\varepsilon_3+(-\nu_3+\nu_4)\varepsilon_4],$$
where
$\nu_i=\Lambda_i+1,\,\,\nu_i\in\mathbb{N},\,\,i=1,2,3,4.$

By formula (\ref{lam}), eigenvalue $\lambda(\Lambda)$, corresponding to highest weight
$\Lambda$, is equal to
\begin{equation}
\label{d1}
\lambda(\Lambda)=-\frac{1}{12\gamma}[\langle\Lambda+\beta,\Lambda+\beta\rangle-\langle\beta,\beta\rangle]=
\end{equation}
$$=-\frac{1}{48\gamma}\left[(2\nu_1+2\nu_2+\nu_3+\nu_4)^2+(2\nu_2+\nu_3+\nu_4)^2+2\nu_3^2+2\nu_4^2-56\right].$$

Calculation by formula (\ref{dim}) of dimension $d(\Lambda+\beta)$ of representation
$\rho(\Lambda)$ corresponding to highest weight $\Lambda$ gives
$$d(\Lambda+\beta)=\frac{(\Lambda+\beta,\varepsilon_1-\varepsilon_2)}{(\beta,\varepsilon_1-\varepsilon_2)}
\cdot\frac{(\Lambda+\beta,\varepsilon_1-\varepsilon_3)}{(\beta,\varepsilon_1-\varepsilon_3)}\cdot
\frac{(\Lambda+\beta,\varepsilon_1-\varepsilon_4)}{(\beta,\varepsilon_1-\varepsilon_4)}\cdot
\frac{(\Lambda+\beta,\varepsilon_2-\varepsilon_3)}{(\beta,\varepsilon_2-\varepsilon_3)}$$
$$\times\frac{(\Lambda+\beta,\varepsilon_2-\varepsilon_4)}{(\beta,\varepsilon_2-\varepsilon_4)}\cdot
\frac{(\Lambda+\beta,\varepsilon_3-\varepsilon_4)}{(\beta,\varepsilon_3-\varepsilon_4)}\cdot
\frac{(\Lambda+\beta,\varepsilon_1+\varepsilon_2)}{(\beta,\varepsilon_1+\varepsilon_2)}\cdot
\frac{(\Lambda+\beta,\varepsilon_1+\varepsilon_3)}{(\beta,\varepsilon_1+\varepsilon_3)}$$
$$\times\frac{(\Lambda+\beta,\varepsilon_1+\varepsilon_4)}{(\beta,\varepsilon_1+\varepsilon_4)}\cdot
\frac{(\Lambda+\beta,\varepsilon_2+\varepsilon_3)}{(\beta,\varepsilon_2+\varepsilon_3)}\cdot
\frac{(\Lambda+\beta,\varepsilon_2+\varepsilon_4)}{(\beta,\varepsilon_2+\varepsilon_4)}\cdot
\frac{(\Lambda+\beta,\varepsilon_3+\varepsilon_4)}{(\beta,\varepsilon_3+\varepsilon_4)}.$$
On the ground of (\ref{m1}) and (\ref{m2}),
$$\frac{(\Lambda+\beta,\varepsilon_1-\varepsilon_2)}{(\beta,\varepsilon_1-\varepsilon_2)}=
\nu_1,\quad
\frac{(\Lambda+\beta,\varepsilon_1-\varepsilon_3)}{(\beta,\varepsilon_1-\varepsilon_3)}=
\frac{\nu_1+\nu_2}{2},$$
$$\frac{(\Lambda+\beta,\varepsilon_1-\varepsilon_4)}{(\beta,\varepsilon_1-\varepsilon_4)}=
\frac{\nu_1+\nu_2+\nu_3}{3},\quad
\frac{(\Lambda+\beta,\varepsilon_2-\varepsilon_3)}{(\beta,\varepsilon_2-\varepsilon_3)}=
\nu_2,$$
$$\frac{(\Lambda+\beta,\varepsilon_2-\varepsilon_4)}{(\beta,\varepsilon_2-\varepsilon_4)}=
\frac{\nu_2+\nu_3}{2},\quad
\frac{(\Lambda+\beta,\varepsilon_3-\varepsilon_4)}{(\beta,\varepsilon_3-\varepsilon_4)}=
\nu_3,$$
$$\frac{(\Lambda+\beta,\varepsilon_1+\varepsilon_2)}{(\beta,\varepsilon_1+\varepsilon_2)}=
\frac{\nu_1+2\nu_2+\nu_3+\nu_4}{5},\quad
\frac{(\Lambda+\beta,\varepsilon_1+\varepsilon_3)}{(\beta,\varepsilon_1+\varepsilon_3)}=
\frac{\nu_1+\nu_2+\nu_3+\nu_4}{4},$$
$$\frac{(\Lambda+\beta,\varepsilon_1+\varepsilon_4)}{(\beta,\varepsilon_1+\varepsilon_4)}=
\frac{\nu_1+\nu_2+\nu_4}{3},\quad
\frac{(\Lambda+\beta,\varepsilon_2+\varepsilon_3)}{(\beta,\varepsilon_2+\varepsilon_3)}=
\frac{\nu_2+\nu_3+\nu_4}{3},$$
$$\frac{(\Lambda+\beta,\varepsilon_2+\varepsilon_4)}{(\beta,\varepsilon_2+\varepsilon_4)}=
\frac{\nu_2+\nu_4}{2},\quad
\frac{(\Lambda+\beta,\varepsilon_3+\varepsilon_4)}{(\beta,\varepsilon_3+\varepsilon_4)}=
\nu_4.$$
Therefore
\begin{equation}
\label{d2}
d(\Lambda+\beta)=\frac{1}{4320}\nu_1\nu_2\nu_3\nu_4(\nu_1+\nu_2)(\nu_2+\nu_3)(\nu_2+\nu_4)(\nu_1+\nu_2+\nu_3)\times
\end{equation}
$$\times(\nu_1+\nu_2+\nu_4)(\nu_2+\nu_3+\nu_4)
(\nu_1+\nu_2+\nu_3+\nu_4)(\nu_1+2\nu_2+\nu_3+\nu_4).$$

5) Simple roots and fundamental weights of the Lie algebra $\mathfrak{so}(8)$ indicated above generate respective lattices $\Lambda_0(D_4)$
and $\Lambda_1(D_4)$, i.e.
\begin{equation}
\label{d3}
\Lambda_0(D_4)=\left\{\sum\limits_{i=1}^{4}\Psi_i\alpha_i\,\mid\,\Psi_i\in\mathbb{Z},\,\,i=1,2,3,4\right\},
\end{equation}
$$\Lambda_1(D_4)=\left\{\sum\limits_{i=1}^{4}\Lambda_i\bar{\omega}_i\,\mid\,\Lambda_i\in\mathbb{Z},\,\,i=1,2,3,4\right\}.$$

After expressing roots via fundamental weights
$$\alpha_1=2\bar{\omega}_1-\bar{\omega}_2,\,\,\alpha_2=-\bar{\omega}_1+2\bar{\omega}_2-\bar{\omega}_3-\bar{\omega}_4,\,\,
\alpha_3=-\bar{\omega}_2+2\bar{\omega}_3,\,\,\alpha_4=-\bar{\omega}_2+2\bar{\omega}_4$$
and the change of variables
$$\left\{
\begin{array}{l}
\Psi_1=\Omega_2+\Omega_3+\Omega_4, \\
\Psi_2=-\Omega_1+2\Omega_2+2\Omega_3+2\Omega_4, \\
\Psi_3=-\Omega_1+\Omega_2+2\Omega_3+\Omega_4, \\
\Psi_4=-\Omega_1+\Omega_2+\Omega_3+2\Omega_4
\end{array}\right.\quad\Leftrightarrow\quad
\left\{
\begin{array}{l}
\Omega_1=2\Psi_1-\Psi_2, \\
\Omega_2=-\Psi_1+2\Psi_2-\Psi_3-\Psi_4, \\
\Omega_3=\Psi_1-\Psi_2+\Psi_3, \\
\Omega_4=\Psi_1-\Psi_2+\Psi_4
\end{array}\right.$$
the lattice $\Lambda_0(D_4)$ takes the following view
\begin{equation}
\label{d4}
\Lambda_0(D_4)=\{\Omega_1(\bar{\omega}_1-\bar{\omega}_3-\bar{\omega}_4)+\Omega_2\bar{\omega}_2+2\Omega_3\bar{\omega}_3+
2\Omega_4\bar{\omega}_4\,\,\mid\,\Omega_i\in\mathbb{Z},\,\,i=1,2,3,4\}.
\end{equation}
Also under the change of variables $\{\Lambda_1=\Omega_1,\,\Lambda_2=\Omega_2,\,\Lambda_3=\Omega_3-\Omega_1,\,\Lambda_4=\Omega_4-\Omega_1\}$
the lattice $\Lambda_1(D_4)$ in base  $\{\bar{\omega}_1-\bar{\omega}_3-\bar{\omega}_4,\bar{\omega}_2,\bar{\omega}_3,\bar{\omega}_4\}$
takes the following form
\begin{equation}
\label{d5}
\Lambda_1(D_4)=\{\Omega_1(\bar{\omega}_1-\bar{\omega}_3-\bar{\omega}_4)+\Omega_2\bar{\omega}_2+\Omega_3\bar{\omega}_3+
\Omega_4\bar{\omega}_4\,\,\mid\,\Omega_i\in\mathbb{Z},\,\,i=1,2,3,4\}.
\end{equation}

It follows from  (\ref{d4}) and (\ref{d5}) that $\Lambda_0(D_4)\subset\Lambda_1(D_4)$ and $\Lambda_1(D_4)/\Lambda_0(D_4)\cong\mathbb{Z}_2\oplus\mathbb{Z}_2$.
Therefore there exist five lattices  $\Lambda$ satisfying the relation $\Lambda_0(D_4)\subseteq\Lambda\subseteq\Lambda_1(D_4)$, namely,
$\Lambda_0(D_4)$, $\Lambda_{1/4}(D_4)$, $\Lambda_{1/2}(D_4)$, $\Lambda_{3/4}(D_4)$, $\Lambda_1(D_4)$, where
$$\Lambda_{1/4}(D_4)=\{\Omega_1(\bar{\omega}_1-\bar{\omega}_3-\bar{\omega}_4)+\Omega_2\bar{\omega}_2+2\Omega_3\bar{\omega}_3+
\Omega_4\bar{\omega}_4\,\,\mid\,\Omega_i\in\mathbb{Z},\,\,i=1,2,3,4\};$$
$$\begin{array}{r c l}
&&\Lambda_{1/2}(D_4)=\{\Omega_1(\bar{\omega}_1-\bar{\omega}_3-\bar{\omega}_4)+\Omega_2\bar{\omega}_2+\Omega_3\bar{\omega}_3+
\Omega_4\bar{\omega}_4\,\,\mid\\
&&\Omega_3\equiv\Omega_4(\mbox{mod}\,2),\,\,\Omega_i\in\mathbb{Z},\,\,i=1,2,3,4\};
\end{array}$$
$$\Lambda_{3/4}(D_4)=\{\Omega_1(\bar{\omega}_1-\bar{\omega}_3-\bar{\omega}_4)+\Omega_2\bar{\omega}_2+\Omega_3\bar{\omega}_3+
2\Omega_4\bar{\omega}_4\,\,\mid\,\Omega_i\in\mathbb{Z},\,\,i=1,2,3,4\}.$$

Lie groups, corresponding to these characteristic lattices, are given in table \ref{Tab:list}.
Lie groups, corresponding to the lattices $\Lambda_{1/4}(D_4)$, $\Lambda_{1/2}(D_4)$, $\Lambda_{3/4}(D_4)$,
are isomorphic to $\mbox{SO}(8)\cong\mbox{Spin}(8)/\mathbb{Z}_2$.

a) Let us present formulas for Laplacian spectrum of the Lie group $\mbox{Spin}(8)$.

6a) By formula (\ref{d3}) we get the following set of highest weights of the Lie group $\mbox{Spin}(8)$.
$$\Lambda^{+}(\mbox{Spin(8)})=\left\{\sum\limits_{i=1}^{4}\Lambda_i\bar{\omega}_i\,\mid\,\Lambda_i\in\mathbb{Z}_{+},\,\,i=1,2,3,4\right\}.$$

7a) Set $\Lambda=\sum\limits_{i=1}^{4}\Lambda_i\bar{\omega}_i\in\Lambda^{+}(\mbox{Spin}(8))$, then by p.~4) eigenvalue  $\lambda(\Lambda)$ and dimension
 $d(\Lambda+\beta)$ are calculated by formulas  (\ref{d1}) and (\ref{d2}) respectively, where $\nu_i=\Lambda_i+1$, $\nu_i\in\mathbb{N}$, $i=1,2,3,4.$

8a) Applying formula (\ref{abc}) and results of preceding point, we get the following multiplicity of the eigenvalue $\lambda(\Lambda)$
$$\sigma(\Lambda)=\frac{1}{4320^2}\sum_{\Xi_8}\nu_1^2\nu_2^2\nu_3^2\nu_4^2(\nu_1+\nu_2)^2(\nu_2+\nu_3)^2(\nu_2+\nu_4)^2(\nu_1+\nu_2+\nu_3)^2$$
$$\times(\nu_1+\nu_2+\nu_4)^2(\nu_2+\nu_3+\nu_4)^2
(\nu_1+\nu_2+\nu_3+\nu_4)^2(\nu_1+2\nu_2+\nu_3+\nu_4)^2,$$
where
\begin{equation}
\label{xi8}
\begin{array}{r c l}
&&\Xi_5=\{(2\nu_1+2\nu_2+\nu_3+\nu_4)^2+(2\nu_2+\nu_3+\nu_4)^2+2\nu_3^2+\\
&&+2\nu_4^2=56-48\gamma\lambda\,\mid\,\nu_i\in\mathbb{N},\,\,i=1,2,3,4\}.
\end{array}
\end{equation}

The least by modulus non-zero eigenvalue of Laplacian is equal to $-\frac{7}{12\gamma}$ and corresponds to irreducible complex representations
 of the Lie group $\mbox{Spin}(8)$ with highest weights  $\bar{\omega}_1$, $\bar{\omega}_3$ and $\bar{\omega}_4$.
Dimensions of these representations are equal to $8$. Therefore the multiplicity of the eigenvalue $-\frac{7}{12\gamma}$ is equal to $8^2+8^2+8^2=192$.

b) Let us give formulas for Laplacian spectrum of the Lie group
$\mbox{SO}(8)\cong\mbox{Spin}(8)/\mathbb{Z}_2$.

6b) By formula (\ref{d3}), we get the sets of highest weights of the Lie group $\mbox{SO}(8)$
$$\begin{array}{r c l}
&&\Lambda^{+}(\mbox{SO}(8))=\{\Omega_1\bar{\omega}_1+\Omega_2\bar{\omega}_2+(2\Omega_3-\Omega_1)\bar{\omega}_3+(\Omega_4-\Omega_1)\bar{\omega}_4\,\mid\\
&&\Omega_i\in\mathbb{Z},\,i=1,2,3,4;\,\,\Omega_1\geq 0,\,\,\Omega_2\geq 0,\,\,2\Omega_3-\Omega_1\geq 0,\,\,\Omega_4-\Omega_1\geq 0\}.
\end{array}$$

7b) Set
$\Lambda=\Omega_1\bar{\omega}_1+\Omega_2\bar{\omega}_2+(2\Omega_3-\Omega_1)\bar{\omega}_3+(\Omega_4-\Omega_1)\bar{\omega}_4\in\Lambda^{+}(\mbox{SO}(8))$,
then by p.~4) eigenvalue  $\lambda(\Lambda)$ and dimension $d(\Lambda+\beta)$ are calculated by formulas  (\ref{d1}) and (\ref{d2}) respectively, where
$\nu_1=\Omega_1+1,$ $\nu_2=\Omega_2+1,$ $\nu_3=2\Omega_3-\Omega_1+1$, $\nu_4=\Omega_4-\Omega_1+1$; $\nu_i\in\mathbb{N}$, $i=1,2,3,4$;
$\nu_3\equiv\nu_1(\mbox{mod}\,2).$

8b) Applying formula (\ref{abc}) and results of preceding point, we get the following multiplicity of the eigenvalue $\lambda(\Lambda)$
$$\sigma(\Lambda)=\frac{1}{4320^2}\sum_{\Xi_9}\nu_1^2\nu_2^2\nu_3^2\nu_4^2(\nu_1+\nu_2)^2(\nu_2+\nu_3)^2(\nu_2+\nu_4)^2(\nu_1+\nu_2+\nu_3)^2$$
$$\times(\nu_1+\nu_2+\nu_4)^2(\nu_2+\nu_3+\nu_4)^2
(\nu_1+\nu_2+\nu_3+\nu_4)^2(\nu_1+2\nu_2+\nu_3+\nu_4)^2,$$
where
\begin{equation}
\label{xi9}
\begin{array}{r c l}
&&\Xi_6=\{(2\nu_1+2\nu_2+\nu_3+\nu_4)^2+(2\nu_2+\nu_3+\nu_4)^2+2\nu_3^2+\\
&&+2\nu_4^2=56-48\gamma\lambda\,\mid\,\nu_3\equiv\nu_1(\mbox{mod}\,2),\,\,\nu_i\in\mathbb{N},\,\,i=1,2,3,4\}.
\end{array}
\end{equation}

The least by modulus non-zero eigenvalue of Laplacian is equal to $-\frac{7}{12\gamma}$ and corresponds to irreducible complex representation of the Lie group
$\mbox{SO}(8)$ with highest weight $\bar{\omega}_4$.
The dimension of this representation is equal to $8$. Then the multiplicity of the eigenvalue $-\frac{7}{12\gamma}$ is equal to $8^2=64$.

с) Let us give formulas for Laplacian spectrum of the Lie group $\mbox{PSO}(8)\cong\mbox{SO}(8)/\mbox{C}(\mbox{SO}(8))$.

6с) By formula (\ref{d3}), we get the set of highest weights of the Lie group $\mbox{PSO}(8)$
$$\begin{array}{r c l}
&&\Lambda^{+}(\mbox{PSO}(8))=
\{\Omega_1\bar{\omega}_1+\Omega_2\bar{\omega}_2+(2\Omega_3-\Omega_1)\bar{\omega}_3+(2\Omega_4-\Omega_1)\bar{\omega}_4\,\mid\\
&&\Omega_i\in\mathbb{Z},\,i=1,2,3,4;\,\,\Omega_1\geq 0,\,\,\Omega_2\geq 0,\,\,2\Omega_3-\Omega_1\geq 0,\,\,2\Omega_4-\Omega_1\geq 0\}.
\end{array}$$

7с) Set $\Lambda=\Omega_1\bar{\omega}_1+\Omega_2\bar{\omega}_2+(2\Omega_3-\Omega_1)\bar{\omega}_3+(2\Omega_4-\Omega_1)\bar{\omega}_4
\in\Lambda^{+}(\mbox{PSO}(8))$, then by p.~4) eigenvalue $\lambda(\Lambda)$ and
dimension $d(\Lambda+\beta)$ are calculated by formulas (\ref{d1}) and (\ref{d2}) respectively, where
$\nu_1=\Omega_1+1$, $\nu_2=\Omega_2+1$, $\nu_3=2\Omega_3-\Omega_1+1$, $\nu_4=2\Omega_4-\Omega_1+1$, $\nu_i\in\mathbb{N}$, $i=1,2,3,4$,
$\nu_1\equiv\nu_3\equiv\nu_4(\mbox{mod}\,2).$

8с) Applying formula (\ref{abc}) and results of preceding point, we get the following multiplicity of the eigenvalue $\lambda(\Lambda)$
$$\sigma(\Lambda)=\frac{1}{4320^2}\sum_{\Xi_{10}}\nu_1^2\nu_2^2\nu_3^2\nu_4^2(\nu_1+\nu_2)^2(\nu_2+\nu_3)^2(\nu_2+\nu_4)^2(\nu_1+\nu_2+\nu_3)^2$$
$$\times(\nu_1+\nu_2+\nu_4)^2(\nu_2+\nu_3+\nu_4)^2
(\nu_1+\nu_2+\nu_3+\nu_4)^2(\nu_1+2\nu_2+\nu_3+\nu_4)^2,$$
where
\begin{equation}
\label{xi10}
\begin{array}{r c l}
&&\Xi_{7}=\{(2\nu_1+2\nu_2+\nu_3+\nu_4)^2+(2\nu_2+\nu_3+\nu_4)^2+2\nu_3^2+\\
&&+2\nu_4^2=56-48\gamma\lambda\,\mid\,\nu_1\equiv\nu_3\equiv\nu_4(\mbox{mod}\,2),\,\,\nu_i\in\mathbb{N},\,\,i=1,2,3,4\}.
\end{array}
\end{equation}

The least by modulus non-zero eigenvalue of Laplacian is equal to $-\frac{1}{\gamma}$ and corresponds to irreducible complex representation of the Lie group
$\mbox{PSO}(8)$ with highest weight $\bar{\omega}_2$.
The dimension of this representation is equal to $28$. Consequently the multiplicity of the eigenvalue $-\frac{1}{\gamma}$ is equal to $28^2=784$.

\section{Necessary information from number theory}

In this section are presented all necessary information on classical solutions of presentation problem of natural numbers by
values of some positively defined  integer quadratic forms from two, three and four variables on integer vectors, applying in next sections.

\begin{theorem}{\rm \cite{Ven}, \cite{Dev}}.
\label{sqq}
A natural number $k$ can be presented in the form
\begin{equation}
\label{square1}
k=x^2+y^2,\quad x,\,y\in\mathbb{Z},
\end{equation}
if and only if $k$ has no prime factor $p$ with condition $p\equiv 3(\Mod 4)$, which occurs in odd
power into factorization of $k$ by prime factors.

In addition a number $N_2(k)$ of all solutions to the equation (\ref{square1}) is equal to quadruplicate difference of quantities of (natural) divisors $d$ of $k$ such that
$d\equiv 1(\Mod 4)$ and divisors $d$ of $k$ such that $d\equiv 3(\Mod 4)$.
\end{theorem}

\begin{theorem}{\rm \cite{BZS}}
\label{sq1}
Assume that a natural number $k$ can be presented in the form
\begin{equation}
\label{squared21}
k=x^2+y^2,\quad x<y,\,\,x,\,y\in\mathbb{N},
\end{equation}
$L_2(k)$ is the quantity of such presentations. Then

1. If $k\neq m^2$, $k\neq 2m^2$ for any $m\in\mathbb{N}$, then $L_2(k)=N_2(k)/8$.

2. If $k=m^2$ or $k=2m^2$ for some $m\in\mathbb{N}$, then $L_2(k)=\left(N_2(k)-4\right)/8$.

\end{theorem}

Theorems \ref{sqq} and \ref{sq1} imply directly

\begin{corollary}{\rm \cite{BZS}}
\label{cor}
A natural number $k$ can be presented in the form (\ref{squared21}) if and only if the following conditions are fulfilled.

1. $k\geq 5$ and $k$ has no prime divisor $p$ with condition $p\equiv 3(\Mod 4)$, which occurs in odd power into factorization of $k$ by prime factors.

2. If $k=m^2$ or $k=2m^2$ for some $m\in\mathbb{N}$ then the difference of quantities of (natural) divisors $d$ of $k$ such that $d\equiv 1(\Mod 4)$ and divisors $d$ of $k$ such that
$d\equiv 3(\Mod 4)$ is more than 1.
\end{corollary}

\begin{theorem} {\rm \cite{Buh},\,\cite{Ven}}
1. An odd natural number $k$ can be presented in the form
\begin{equation}
\label{square2}
k=x^2+2y^2,\quad x,\,y\in\mathbb{Z},
\end{equation}
where $\mbox{GCD}(x,y)=1$, if and only if the factorization of $k$ by prime factors does not contain prime numbers of the form $8n+5$ и $8n+7$.

2. For any  $k\in\mathbb{N}$ the number $N_{1,2}(k)$ of all solutions to the equation (\ref{square2})
is equal to doubled difference of quantities (of natural) divisors $d$ of $k$ such that
$d\equiv 1({\rm mod}\,8)$ or $d\equiv 3({\rm mod}\,8)$ and divisors $d$ of $k$ such that $d\equiv 5({\rm mod}\,8)$ or $d\equiv 7({\rm mod}\,8)$.
\end{theorem}

\begin{theorem}
\label{j}
Assume that a natural number $k$ can be presented in the form
\begin{equation}
\label{sq}
k=x^2+2y^2,\quad x\neq y,\,\,x,\,y\in\mathbb{N},
\end{equation}
$L_{1,2}(k)$ is the quantity of such presentations. Then

1. If $k\neq\alpha m^2$ for any $m\in\mathbb{N}$, $\alpha=1,2,3$, then $L_{1,2}(k)=N_{1,2}(k)/4$.

2. If $k=m^2$ or $k=2m^2$ for some $m\in\mathbb{N}$, then $L_{1,2}(k)=\left(N_{1,2}(k)-2\right)/4$.

3. If $k=3m^2$ for some $m\in\mathbb{N}$ then $L_{1,2}(k)=\left(N_{1,2}(k)-4\right)/4$.
\end{theorem}

\begin{proof}
1. Let $k\neq\alpha m^2$ for any $m\in\mathbb{N}$, $\alpha=1,2,3$. Then if $(x,y)$ is a presentation of  $k$ in the form (\ref{square2}), then
$x\neq 0$, $y\neq 0$ and $x\neq y$. To every presentation $(x,y)$ of $k$ in the form (\ref{sq})
correspond exactly four  different presentations of $k$ in the form (\ref{square2}), namely, ordered pairs $(\pm x,\pm y)$.
Therefore $N_{1,2}(k)=4L_{1,2}(k)$, whence it follows the required formula.

2. Let $k=m^2$ ($k=2m^2$) for some $m\in\mathbb{N}$. Then, besides presentations of $k$ in the form (\ref{square2}), described in p.~1, there are also
only two different presentations of $k$  in the form (\ref{square2}), namely, ordered pairs $(\pm m,0)$ (respectively, $(0,\pm m)$). Therefore
$N_{1,2}(k)=4L_{1,2}(k)+2$, whence it follows the required formula.

3. Let $k=3m^2$ for some $m\in\mathbb{N}$. Then, besides presentations of $k$ in the form (\ref{square2}), described in p.~1, there are also
only four different presentations of $k$  in the form (\ref{square2}), namely, ordered pairs $(\pm m,\pm m)$. Therefore
$N_{1,2}(k)=4L_{1,2}(k)+4$, whence it follows the required formula.
\end{proof}

Later we shall need the notion of \textit{the Legendre-Jacobi symbol}. There are different definitions of this symbol. Here is given the most simple its definition, taken from \cite{C} and belonging to Russian mathematician E.I.~Zolotarev (1847--1878).

\begin{definition}
Let $n>1$ be an odd natural number, $a$ be an integer number, coprime with $n.$ The Legendre-Jacobi
symbol $\left(\frac{a}{n}\right)$ is the sign of permutation on residue ring $\Mod n$ obtained by multiplication of this ring by $a\Mod n.$
\end{definition}

\begin{theorem}{\rm \cite{Ven}}
Assume that a natural number $k$ can be presented in the form
\begin{equation}
\label{mq1}
k=x^2+3y^2,\quad x,\,y\in\mathbb{Z},
\end{equation}
$N_{1,3}(k)$ is the quantity of such presentations. Set $k=2^{l}m$, where $m$ is an odd number, $l\in\mathbb{Z}_{+}$.
Then

1. If $l=0$ then $N_{1,3}(k)=2\chi(m)$.

2. If $l$ is an odd number then $N_{1,3}(k)=0$.

3. If $l\neq 0$ is an even number then $N_{1,3}(k)=6\chi(m)$.

Here $\chi(m)=\sum\left(\frac{-3}{d}\right)$, in both cases the summation is taken by all divisors $d$ of  $m$.
\end{theorem}

\begin{theorem}
\label{j1}
Assume that a natural number $k$ can be presented in the form
\begin{equation}
\label{q1}
k=x^2+3y^2,\quad x\neq y,\,\,x,\,y\in\mathbb{N},
\end{equation}
$L_{1,3}(k)$ is the quantity of such presentations. Then

1. If $k\neq\alpha m^2$ for any $m\in\mathbb{N}$, $\alpha=1,3,4$, then $L_{1,3}(k)=N_{1,3}(k)/4$.

2. If $k=m^2$ for some odd number $m\in\mathbb{N}$ then
$L_{1,3}(k)=(N_{1,3}(k)-2)/4.$

3. If $k=3m^2$ for some $m\in\mathbb{N}$ then $L_{1,3}(k)=(N_{1,3}(k)-2)/4$.

4. If $k=4m^2$ for some $m\in\mathbb{N}$ then $L_{1,3}(k)=(N_{1,3}(k)-6)/4$.
\end{theorem}

\begin{proof}
1. Let $k\neq\alpha m^2$ for any $m\in\mathbb{N}$, $\alpha=1,3,4$. Then if $(x,y)$ is a presentation of $k$ in the form (\ref{mq1}), then
$x\neq 0$, $y\neq 0$ and $x\neq y$.
To every presentation $(x,y)$ of $k$ in the form (\ref{q1})
correspond exactly four different presentations of $k$ in the form (\ref{mq1}), namely, ordered pairs  $(\pm x,\pm y)$.
Therefore  $N_{1,3}(k)=4L_{1,3}(k)$, whence it follows the required formula.

2. Let $k=m^2$ for some odd number $m\in\mathbb{N}$. Then, besides presentations of $k$ in the form (\ref{mq1}), described in p.~1,
there are also
only two different presentations of $k$  in the form (\ref{mq1}), namely, ordered pairs $(\pm m,0)$. Therefore
$N_{1,3}(k)=4L_{1,3}(k)+2$, whence it follows the required formula.

3. Let $k=3m^2$ for some $m\in\mathbb{N}$. Then, besides presentations of $k$ in the form (\ref{mq1}), described in p.~1,
there are also
only two different presentations of $k$  in the form (\ref{mq1}), namely, ordered pairs $(0,\pm m)$. Therefore
$N_{1,3}(k)=4L_{1,3}(k)+2$, whence it follows the required formula.

4. Let $k=4m^2$ for some $m\in\mathbb{N}$. Then, besides presentations of $k$ in the form (\ref{mq1}), described in p.~1,
there are also only 6 different presentations of $k$  in the form (\ref{mq1}), namely, ordered pairs $(\pm m,\pm m)$, $(\pm 2m,0)$. Therefore
$N_{1,2}(k)=4L_{1,2}(k)+6$, whence it follows the required formula.
\end{proof}

\begin{theorem}{\rm \cite{Dev}}
\label{th}
A natural number $k$ can be presented in the form
\begin{equation}
\label{three}
k=x^2+y^2+z^2,\quad x,\,y,\,z\in\mathbb{Z},
\end{equation}
if and only if $k$ cannot be presented in the form $4^{m}(8l+7)$, where $m,$ $l\in\mathbb{Z}_{+}$.
\end{theorem}

\begin{theorem} {\rm \cite{Ven}}
1. Let $k\in\mathbb{N}$, $k=1, 2({\rm mod}\,4)$, $k\neq 1$.
Then the number $\psi(k)$ of proper presentations of $k$ in the form (\ref{three}) is finite and equal to $12h(k)$, where
\begin{equation}
\label{h}
h(k)=\sum\left(-\frac{k}{a}\right),
\end{equation}
and the summation is taken for all $a$ such that $a\in\mathbb{N}$, $0<a<k$, $a$ is coprime with
$2k$; $\left(-\frac{k}{a}\right)$ is the Legendre-Jacobi symbol.

2. Let $k\in\mathbb{N}$, $k=3({\rm mod}\,8)$, $k\neq 3$. Then the number $\psi(k)$ of proper presentations of $k$ in the form
 (\ref{three}) is finite and equal to $24h^{\prime}(k)$, where
\begin{equation}
\label{hprime}
h^{\prime}(k)=\frac{1}{3}\sum\left(\frac{b}{k}\right),
\end{equation}
and the summation is taken for all $b$ such that $b\in\mathbb{N}$, $0<b<k$, $b$ is coprime with  $2k$; $\left(\frac{b}{k}\right)$ is the Legendre-Jacobi symbol.
\end{theorem}

Let us define function $F(k)$, $k\in\mathbb{N}$, by the following rule

1) If $k=(2m-1)^2$ for some  $m\in\mathbb{N}$, then
$F(k)=\sum h(k/\delta^2)-\frac{1}{2}.$

2) If $k\neq (2m-1)^2$ for any  $m\in\mathbb{N}$, then
$F(k)=\sum h(k/\delta^2).$

In both cases the summation is taken by all square divisors $\delta^2$ of the number $k$.

\begin{theorem} {\rm\cite{Ven}}
Let $k\in\mathbb{N}$, $N_3(k)$ be a number of all presentations of the number $k$ in the form (\ref{three}). Then the following statements are valid.

1. If $k=1, 2({\rm mod}\,4)$, $k\neq 1$, then $N_3(k)=12F(k)$.

2. If $k=3({\rm mod}\,8)$, $k\neq 3$, then  $N_3(k)=8F(k)$.

3. If $k=7({\rm mod}\,8)$ then  $N_3(k)=0$.

4. If $k=0({\rm mod}\,4)$ then $N_3(k)=N_3(k/4)$.
\end{theorem}

\begin{theorem}{\rm\cite{BZS}}
\label{3sq}
Assume that a natural number $k$ can be presented in the form
\begin{equation}
\label{squared3}
k=x^2+y^2+z^2,\quad x<y<z,\,\,x,\,y,\,z\in\mathbb{N},
\end{equation}
$L_3(k)$ is the quantity of such presentations. Then

1. If $k\neq\alpha m^2$ for any $m\in\mathbb{N}$, $\alpha=1,2,3$, then
$$L_3(k)=\frac{N_3(k)-3N_2(k)-6N_{1,2}(k)}{48}.$$

2. If $k=m^2$ for some $m\in\mathbb{N}$ then
$$L_3(k)=\frac{N_3(k)-3N_2(k)-6N_{1,2}(k)+18}{48}.$$

3. If $k=2m^2$ for some $m\in\mathbb{N}$ then
$$L_3(k)=\frac{N_3(k)-3N_2(k)-6N_{1,2}(k)+12}{48}.$$

4. If $k=3m^2$ for some $m\in\mathbb{N}$ then
$$L_3(k)=\frac{N_3(k)-6N_{1,2}(k)+16}{48}.$$
\end{theorem}

\begin{theorem}\label{r}
1. A natural number $k$ can be presented in the form
\begin{equation}
\label{s112}
k=x^2+y^2+2z^2,\quad x,\,y,\,z\in\mathbb{Z},
\end{equation}
if and only if a number $2k$ can't be presented in the form $4^m(8l+1)$, where $m$, $l\in\mathbb{Z}_{+}$.

2. The quantity $N_{1,1,2}(k)$ of all presentations of  $k$ in the form (\ref{s112}) is equal to $N_3(2k)$ for even number  $k$  and $N_3(2k)/3$
for odd number $k$.
\end{theorem}

\begin{proof}
 We shall use the formula
\begin{equation}
\label{op}
2(x^2+y^2+2z^2)=(x+y)^2+(x-y)^2+(2z)^2,\quad x,\,y,\,z\in\mathbb{R}.
\end{equation}

In consequence of (\ref{op}), to every presentation $(x,y,z)$ of $k$ in the form (\ref{s112}) correspond  the presentation $(x+y,x-y,2z)$ of $2k$ in the form
 (\ref{three}).
Moreover, if $k$ is an even number and $(x,y,z)$ is a presentation of $2k$ in the form (\ref{three}), then $x,\,y,\,z$ are even numbers.
On the ground of (\ref{op}), to this presentation correspond  the presentation $\left(\frac{x+y}{2},\frac{x-y}{2},\frac{z}{2}\right)$
of $k$ in the form (\ref{s112}).
Therefore $N_{1,1,2}(k)=N_3(2k)$, if $k$ is an even number.

If $k$ is an odd number and $(x,y,z)$ is a presentation of $2k$ in the form (\ref{three}), then
there is only one even number among them. We can suppose for definiteness that
 $z$ is an even number. If $x\neq y$
then on the ground of (\ref{op}), to mutually different presentations $(x,y,z)$, $(x,z,y)$, $(y,x,z)$, $(y,z,x)$, $(z,x,y)$, $(z,y,x)$
correspond exactly two different presentations $\left(\frac{x+y}{2},\frac{x-y}{2},\frac{z}{2}\right)$, $\left(\frac{x-y}{2},\frac{x+y}{2},\frac{z}{2}\right)$
of $k$ in the form (\ref{s112}).
If $x=y$, then to mutually different presentations $(x,x,z)$, $(x,z,x)$, $(z,x,x)$
correspond exactly one presentation $\left(x,0,\frac{z}{2}\right)$ of $k$ in the form (\ref{s112}).
Therefore $N_{1,1,2}(k)=N_3(2k)/3$, if $k$ is an odd number.

It's remain to apply theorem \ref{th}.
\end{proof}

\begin{theorem}\label{ssq}
Assume that a natural number $k$ can be presented in the form
\begin{equation}
\label{sq112}
k=x^2+y^2+2z^2,\quad x<y,\,x\neq z,\,y\neq z,\,x,\,y,\,z\in\mathbb{N},
\end{equation}
$L_{1,1,2}(k)$ is the quantity of such presentations. Then

1. If $k\neq\alpha m^2$ for any $m\in\mathbb{N}$, $\alpha=1,2,3,4$, then
$$L_{1,1,2}(k)=\frac{N_{1,1,2}(k)-8L_2(k)-8L_{1,2}(k)-16L_{1,3}(k)-16L_2(k/2)}{16}.$$

2. If $k=m^2$ for some odd number $m\in\mathbb{N}$ then
$$L_{1,1,2}(k)=\frac{N_{1,1,2}(k)-8L_2(k)-8L_{1,2}(k)-16L_{1,3}(k)-16L_2(k/2)-4}{16}.$$

3. If $k=2m^2$ for some $m\in\mathbb{N}$ then
$$L_{1,1,2}(k)=\frac{N_{1,1,2}(k)-8L_2(k)-8L_{1,2}(k)-16L_{1,3}(k)-16L_2(k/2)-6}{16}.$$

4. If $k=3m^2$ for some $m\in\mathbb{N}$ then
$$L_{1,1,2}(k)=\frac{N_{1,1,2}(k)-8L_2(k)-8L_{1,2}(k)-16L_{1,3}(k)-16L_2(k/2)-8}{16}.$$

5. If $k=4m^2$ for some $m\in\mathbb{N}$ then
$$L_{1,1,2}(k)=\frac{N_{1,1,2}(k)-8L_2(k)-8L_{1,2}(k)-16L_{1,3}(k)-16L_2(k/2)-12}{16}.$$
\end{theorem}

\begin{proof}
1. Let $k\neq\alpha m^2$ for any $m\in\mathbb{N}$, $\alpha=1,2,3,4$.
Then $k$ has $N_{1,1,2}(k)$ presentations in the form (\ref{s112}), $L_{1,1,2}(k)$ presentations in the form (\ref{sq112}),
$L_2(k)$ presentations in the form (\ref{squared21}), $L_{1,2}(k)$ presentations in the form (\ref{sq}),
$L_{1,3}(k)$ presentations in the form (\ref{q1}).

To every presentation $(x,y,z)$ of $k$ in the form (\ref{sq112}) correspond exactly $16$ different presentations of $k$ in the form (\ref{s112}),
namely, ordered triples
 $(\pm x,\pm y,\pm z)$, $(\pm y,\pm x,\pm z)$.

To every presentation $(x,y)$ of $k$ in the form (\ref{squared21}) correspond exactly $8$ different presenta-
\linebreak tions of $k$ in the form (\ref{s112}),
namely, ordered triples  $(\pm x,\pm y,0)$, $(\pm y,\pm x,0)$.

To every presentation $(x,y)$ of $k$ in the form  (\ref{sq}) correspond exactly $8$ different presenta-
\linebreak tions of $k$ in the form (\ref{s112}),
namely, ordered triples  $(\pm x,0,\pm y)$, $(0,\pm x,\pm y)$.

To every presentation $(x,y)$ of $k$ in the form (\ref{q1}) correspond exactly $16$ different presentations of $k$ in the form
(\ref{s112}), namely, ordered triples $(\pm x,\pm y,\pm y)$, $(\pm y,\pm x,\pm y)$.

Set $L_2(k/2)=0$, if $k$ is an odd number.
If $k$ is an even number, then to every presentation $(x,y)$ of $k/2$ in the form (\ref{squared21}) correspond exactly
 $16$ different presentations of $k$ in the form (\ref{s112}), namely, ordered triples  $(\pm x,\pm x,\pm y)$, $(\pm y,\pm y,\pm x)$.

Thus
$$N_{1,1,2}(k)=16L_{1,1,2}(k)+8L_2(k)+8L_{1,2}(k)+16L_{1,3}(k)+16L_2(k/2),$$
whence it follows the required formula.

2. Let $k=m^2$ for some odd number $m\in\mathbb{N}$. Then, besides presentations of $k$ in the form  (\ref{s112}),
described in p.~1, there are also only 4 different presentations of $k$ in the form
(\ref{s112}), namely, ordered triples $(\pm m,0,0)$, $(0,\pm m,0)$.
Then on the ground of p.~1
$$N_{1,1,2}(k)=16L_{1,1,2}(k)+8L_2(k)+8L_{1,2}(k)+16L_{1,3}(k)+16L_2(k/2)+4,$$
whence it follows the required formula.

3. Let $k=2m^2$ for some $m\in\mathbb{N}$. Then, besides presentations of $k$ in the form  (\ref{s112}),
described in p.~1, there are also only 6 different presentations of $k$ in the form
(\ref{s112}), namely, ordered triples $(\pm m,\pm m,0)$, $(0,0,\pm m)$. Then on the ground of p.~1
$$N_{1,1,2}(k)=16L_{1,1,2}(k)+8L_2(k)+8L_{1,2}(k)+16L_{1,3}(k)+16L_2(k/2)+6,$$
whence it follows the required formula.

4. Let $k=3m^2$ for some  $m\in\mathbb{N}$. Then, besides presentations of $k$ in the form  (\ref{s112}),
described in p.~1, there are also only 6 different presentations of $k$ in the form
 (\ref{s112}), namely, ordered triples  $(\pm m,0,\pm m)$, $(0,\pm m,\pm m)$. Then on the ground of p.~1
$$N_{1,1,2}(k)=16L_{1,1,2}(k)+8L_2(k)+8L_{1,2}(k)+16L_{1,3}(k)+16L_2(k/2)+8,$$
whence it follows the required formula.

5. Let $k=4m^2$ for some $m\in\mathbb{N}$.  Then, besides presentations of $k$ in the form  (\ref{s112}),
described in p.~1, there are also only 12 different presentations of $k$ in the form
 (\ref{s112}), namely, ordered triples  $(\pm m,\pm m,\pm m)$, $(\pm 2m,0,0)$, $(0,\pm 2m,0)$.
Then on the ground of p.~1
$$N_{1,1,2}(k)=16L_{1,1,2}(k)+8L_2(k)+8L_{1,2}(k)+16L_{1,3}(k)+16L_2(k/2)+12,$$
whence it follows the required formula.
\end{proof}

\begin{theorem}{\rm \cite{Dev}}
\label{4sq}
1. Every natural number $k$ can be presented in the form
\begin{equation}
\label{ff}
k=x^2+y^2+z^2+v^2,\quad x,\,y,\,z,\,v\in\mathbb{Z}.
\end{equation}

2. Let $\sigma(k)$ be a sum of all odd divisors of the number $k$.
The quantity $N_4(k)$ of all presentations of  $k$ in the form (\ref{ff}) is equal to $24\sigma(k)$, if
$k$ is an even number or $8\sigma(k)$, if $k$ is an odd number.
\end{theorem}

\begin{theorem}
\label{4sqnew}
Assume that a natural number $k$ can be presented in the form
\begin{equation}
\label{squared4}
k=x^2+y^2+z^2+v^2,\quad x<y<z<v,\,\,x,\,y,\,z,\,v\in\mathbb{N},
\end{equation}
$L_4(k)$ is the quantity of such presentations. Then

1. If $k\neq\alpha m^2$ for any $m\in\mathbb{N}$, $\alpha=1,2,3,4$, then
$$L_4(k)=\frac{N_4(k)-192L_{1,1,2}(k)-192L_3(k)-48L_2(k)-96L_{1,2}(k)-64L_{1,3}(k)-96L_2(k/2)}{384}.$$

2. If $k=m^2$ for some odd number $m\in\mathbb{N}$ then
$$L_4(k)=\frac{N_4(k)-192L_{1,1,2}(k)-192L_3(k)-48L_2(k)-96L_{1,2}(k)-64L_{1,3}(k)-96L_2(k/2)-8}{384}.$$

3. If $k=2m^2$ or $k=4m^2$ for some $m\in\mathbb{N}$ then
$$L_4(k)=\frac{N_4(k)-192L_{1,1,2}(k)-192L_3(k)-48L_2(k)-96L_{1,2}(k)-64L_{1,3}(k)-96L_2(k/2)-24}{384}.$$

4. If $k=3m^2$ for some $m\in\mathbb{N}$ then
$$L_4(k)=\frac{N_4(k)-192L_{1,1,2}(k)-192L_3(k)-48L_2(k)-96L_{1,2}(k)-64L_{1,3}(k)-96L_2(k/2)-32}{384}.$$
\end{theorem}

\begin{proof}
1. Let $k\neq\alpha m^2$ for any $m\in\mathbb{N}$, $\alpha=1,2,3,4$. Then $k$ has $N_4(k)$ presentations in the form (\ref{ff}),
$L_4(k)$ presentations in the form (\ref{squared4}), $L_3(k)$ presentations in the form  (\ref{squared3}),
$L_{1,1,2}(k)$ presentations in the form  (\ref{sq112}), $L_2(k)$ presentations in the form  (\ref{squared21}),
$L_{1,2}(k)$ presentations in the form  (\ref{sq}), $L_{1,3}(k)$ presentations in the form  (\ref{q1}).

To every presentation $(x,y,z,v)$ of $k$ in the form
 (\ref{squared4}) correspond exactly $4!\cdot 2^4=384$ different presentations of $k$ in the form   (\ref{ff}), namely,
 all ordered quaternaries
\linebreak $(\pm u_1,\pm u_2,\pm u_3,\pm u_4)$, where $u_1,\,u_2,\,u_3,\,u_4\in\{x,y,z,v\}$  are mutually different.

To every presentation  $(x,y,z)$ of $k$ in the form
 (\ref{sq112}) correspond exactly  $\frac{4!}{2!}\cdot 2^4=192$ different presentations of $k$ in the form   (\ref{ff}),
namely,  ordered quaternaries  $(\pm x,\pm y,\pm z,\pm z)$, $(\pm y,\pm x,\pm z,\pm z)$, $(\pm z,\pm z,\pm x,\pm y)$,
$(\pm z,\pm z,\pm y,\pm x)$, $(\pm z,\pm x,\pm y,\pm z)$, $(\pm z,\pm y,\pm x,\pm z)$, $(\pm x,\pm z,\pm z,\pm y)$,
$(\pm y,\pm z,\pm z,\pm x)$, $(\pm z,\pm x,\pm z,\pm y)$, $(\pm z,\pm y,\pm z,\pm x)$, $(\pm x,\pm z,\pm y,\pm z)$,
$(\pm y,\pm z,\pm x,\pm z)$.

To every presentation  $(x,y,z)$ of $k$ in the form (\ref{squared3}) correspond exactly $4!\cdot 2^3=192$ different presentations of  $k$ in the form  (\ref{ff}),
namely, all ordered quaternaries $(\pm u_1,\pm u_2,\pm u_3,0)$, $(\pm u_1,\pm u_2,0,\pm u_3)$, $(\pm u_1,0,\pm u_2,\pm u_3)$,
$(0,\pm u_1,\pm u_2,\pm u_3)$, where $u_1,\,u_2,\,u_3\in\{x,y,z\}$ are mutually different.

To every presentation  $(x,y)$ of $k$ in the form
(\ref{squared21}) correspond exactly  $\frac{4!}{2!}\cdot 2^2=48$ different presentations of  $k$ in the form (\ref{ff}), namely,
 all ordered  quaternaries 
$(\pm u_1,\pm u_2,0,0)$, $(\pm u_1,0,\pm u_2,0)$, $(\pm u_1,0,0,\pm u_2)$, $(0,\pm u_1,\pm u_2,0)$,  $(0,\pm u_1,0,\pm u_2)$, $(0,0,\pm u_1,\pm u_2)$,
where $u_1,\,u_2\in\{x,y\}$ и $u_1\neq u_2$.

To every presentation   $(x,y)$ of $k$ in the form
 (\ref{sq}) correspond exactly  $\frac{4!}{2!}\cdot 2^3=96$ different presentations of  $k$ in the form (\ref{ff}), namely,
 ordered  quaternaries
$(\pm x,0,\pm y,\pm y)$, $(0,\pm x,\pm y,\pm y)$, $(\pm y,\pm y,\pm x,0)$,
$(\pm y,\pm y,0,\pm x)$, $(\pm y,\pm x,0,\pm y)$, $(\pm y,0,\pm x,\pm y)$, 
\linebreak $(\pm x,\pm y,\pm y,0)$,
$(0,\pm y,\pm y,\pm x)$, $(\pm y,\pm x,\pm y,0)$, $(\pm y,0,\pm y,\pm x)$,
$(\pm x,\pm y,0,\pm y)$,
\linebreak  $(0,\pm y,\pm x,\pm y)$.

To every presentation   $(x,y)$ of $k$ in the form
(\ref{q1}) correspond exactly   $4\cdot 2^4=64$ different presentations of  $k$ in the form  (\ref{ff}), namely, ordered  quaternaries
$(\pm x,\pm y,\pm y,\pm y)$, 
\linebreak $(\pm y,\pm x,\pm y,\pm y)$, $(\pm y,\pm y,\pm x,\pm y)$, $(\pm y,\pm y,\pm y,\pm x)$.

Set $L_2(k/2)=0$, if $k$ is an odd number. If $k$ is an even number, then to every presentation $(x,y)$ of $k/2$ in the form (\ref{squared21})
correspond exactly  $6\cdot 2^4=96$ different presentations of  $k$ in the form (\ref{ff}), namely, ordered  quaternaries $(\pm x,\pm x,\pm y,\pm y)$,
\linebreak $(\pm x,\pm y,\pm x,\pm y)$, $(\pm x,\pm y,\pm y,\pm x)$, $(\pm y,\pm y,\pm x,\pm x)$, $(\pm y,\pm x,\pm x,\pm y)$, 
 $(\pm y,\pm x,\pm y,\pm x)$.

Thus
$$N_4(k)=384L_4(k)+192L_{1,1,2}(k)+192L_3(k)+48L_2(k)+96L_{1,2}(k)+64L_{1,3}(k)+96L_2(k/2),$$
whence it follows the required formula.

2. Let $k=m^2$ for some odd number $m\in\mathbb{N}$.
Then, besides presentations of $k$ in the form (\ref{ff}), described in p~1, there are also only $8$ different presentations of  $k$ in the form  (\ref{ff}),
namely, ordered sequences  $(\pm m,0,0,0)$, $(0,\pm m,0,0)$, $(0,0,\pm m,0)$, $(0,0,0,\pm m)$. Then on the ground of p.~1
$$N_4(k)=384L_4(k)+192L_{1,1,2}(k)+192L_3(k)+48L_2(k)+96L_{1,2}(k)+64L_{1,3}(k)+96L_2(k/2)+8,$$
whence it follows the required formula.

3. Let $k=2m^2$ for some $m\in\mathbb{N}$. Then, besides presentations of $k$ in the form  (\ref{ff}),
described in p~1, there are also only $24$ different  presentations of  $k$ in the form (\ref{ff}), namely,
ordered sequences  $(\pm m,\pm m,0,0)$, $(\pm m,0,\pm m,0)$, $(\pm m,0,0,\pm m)$,
$(0,\pm m,\pm m,0)$, 
\linebreak $(0,\pm m,0,\pm m)$, $(0,0,\pm m,\pm m)$.

Let $k=4m^2$ for some $m\in\mathbb{N}$. Then, besides presentations of $k$ in the form (\ref{ff}), described in p~1,
there are also only  $24$ different  presentations of  $k$ in the form  (\ref{ff}), namely, ordered sequences  $(\pm m,\pm m,\pm m,\pm m)$,
$(\pm 2m,0,0,0)$, $(0,\pm 2m,0,0)$, $(0,0,\pm 2m,0)$, 
\linebreak $(0,0,0,\pm 2m)$.

Therefore on the ground of p.~1
$$N_4(k)=384L_4(k)+192L_{1,1,2}(k)+192L_3(k)+48L_2(k)+96L_{1,2}(k)+64L_{1,3}(k)+96L_2(k/2)+24,$$
whence it follows the required formula.

4. Let $k=3m^2$ for some $m\in\mathbb{N}$. Then, besides presentations of $k$ in the form  (\ref{ff}), described in p~1,
there are also only $32$ different  presentations of  $k$ in the form (\ref{ff}), namely, ordered sequences  $(\pm m,\pm m,\pm m,0)$,
$(\pm m,0,\pm m,\pm m)$, $(\pm m,\pm m,0,\pm m)$, $(0,\pm m,\pm m,\pm m)$. Then on the ground of p.~1
$$N_4(k)=384L_4(k)+192L_{1,1,2}(k)+192L_3(k)+48L_2(k)+96L_{1,2}(k)+64L_{1,3}(k)+96L_2(k/2)+32,$$
whence it follows the required formula.
\end{proof}

Next theorem follows from theorems \ref{sq1}, \ref{j}, \ref{j1}, \ref{3sq}, \ref{ssq}, \ref{4sqnew}.

\begin{theorem}
\label{5sqnew}
1. If $k\neq\alpha m^2$ for any $m\in\mathbb{N}$, $\alpha=1,2,3,4$, then
$$L_4(k)=\frac{N_4(k)-12N_{1,1,2}(k)-4N_3(k)+18N_2(k)+24N_{1,2}(k)+32N_{1,3}(k)+12N_2(k/2)}{384}.$$

2. If $k=m^2$ for some odd number $m\in\mathbb{N}$ then
$$L_4(k)=\frac{N_4(k)-12N_{1,1,2}(k)-4N_3(k)+18N_2(k)+24N_{1,2}(k)+32N_{1,3}(k)+12N_2(k/2)-120}{384}.$$

3. If $k=2m^2$ for some $m\in\mathbb{N}$ then
$$L_4(k)=\frac{N_4(k)-12N_{1,1,2}(k)-4N_3(k)+18N_2(k)+24N_{1,2}(k)+32N_{1,3}(k)+12N_2(k/2)-72}{384}.$$

4. If $k=3m^2$ for some $m\in\mathbb{N}$ then
$$L_4(k)=\frac{N_4(k)-12N_{1,1,2}(k)-4N_3(k)+18N_2(k)+24N_{1,2}(k)+32N_{1,3}(k)+12N_2(k/2)-64}{384}.$$

5. If $k=4m^2$ for some $m\in\mathbb{N}$ then
$$L_4(k)=\frac{N_4(k)-12N_{1,1,2}(k)-4N_3(k)+18N_2(k)+24N_{1,2}(k)+32N_{1,3}(k)+12N_2(k/2)-216}{384}.$$
\end{theorem}

\section{Conclusion}

In tables \ref{t1},\,\ref{t2},\,\ref{t3},\,\ref{t4},\,\ref{t7},\,\ref{t8} (\ref{t5} and \ref{t6}) we present (on decrease) $10$ 
maximal negative eigenvalues of Laplacian on Lie groups ${\rm Spin}(9)$,  ${\rm SO}(9)$, ${\rm Sp}(4)$, ${\rm PSp}(4)$,
${\rm SO}(8)$, ${\rm PSO}(8)$ (${\rm Spin}(8)$) respectively.

\begin{theorem}
\label{x1}
Let $G={\rm Spin}(9)$ is supplied by biinvariant Riemannian metric $\nu$ such that $\nu(e)=-k_{{\rm ad}}$.
A number $\lambda < 0$ is eigenvalue of Laplacian on $(G,\nu)$ in one and only one of the following cases (I or II)

I. 1) $-7\lambda\in\mathbb{N}$;

2) natural number  $84-56\lambda$ is a sum of squares of four mutually different odd natural numbers, i.e.
 $L_4(21-14\lambda)<L_4(84-56\lambda)$.

In addition the number of highest weights $\Lambda$, such that $\lambda(\Lambda)=\lambda$, is equal to  $L_4(84-56\lambda)$.

II. 1) $-14\lambda\in\mathbb{N}$;

2) natural number  $21-14\lambda$ is a sum of squares of four mutually different natural numbers, i.e. $L_4(21-14\lambda)>0$.

3) natural number $84-56\lambda$ can't be presented as a sum of squares of four mutually different odd natural numbers,
i.e. $L_4(84-56\lambda)=L_4(21-14\lambda)$.

In addition the number of highest weights $\Lambda$, such that $\lambda(\Lambda)=\lambda$, is equal to  $L_4(21-14\lambda)$.
\end{theorem}

\begin{table}[h]
\begin{center}
\begin{tabular}{|c|c|c|c|c|c|c|c|}
\hline
$\No$&$\lambda$&$k_1=21-14\lambda$&$k_2=84-56\lambda$&$L_4(k_1)$&$L_4(k_2)$&$(\nu_1,\nu_2,\nu_3,\nu_4)$&$\sigma(\lambda)$\\
\hline
$1$&$-\frac{4}{7}$&$29$&$116$&$0$&$1$&$(2,1,1,1)$&$81$\\
\hline
$2$&$-\frac{9}{14}$&$30$&$120$&$1$&$1$&$(1,1,1,2)$&$256$\\
\hline
$3$&$-\frac{9}{7}$&$39$&$156$&$1$&$3$&$(3,1,1,1)$&$25376$\\
$ $&$ $&$ $&$ $&$ $&$ $&$(1,1,2,1)$&$ $\\
$ $&$ $&$ $&$ $&$ $&$ $&$(2,1,1,2)$&$ $\\
\hline
$4$&$-\frac{10}{7}$&$41$&$164$&$0$&$1$&$(1,1,1,3)$&$15876$\\
\hline
$5$&$-\frac{12}{7}$&$45$&$180$&$0$&$1$&$(2,2,1,1)$&$53361$\\
\hline
$6$&$-\frac{25}{14}$&$46$&$184$&$1$&$1$&$(1,2,1,2)$&$186624$\\
\hline
$7$&$-2$&$49$&$196$&$0$&$1$&$(2,1,2,1)$&$352836$\\
\hline
$8$&$-\frac{29}{14}$&$50$&$200$&$1$&$1$&$(3,1,1,2)$&$331776$\\
\hline
$9$&$-\frac{15}{7}$&$51$&$204$&$1$&$3$&$(4,1,1,1)$&$1467936$\\
$ $&$ $&$ $&$ $&$ $&$ $&$(2,1,1,3)$&$ $\\
$ $&$ $&$ $&$ $&$ $&$ $&$(1,1,2,2)$&$ $\\
\hline
$10$&$-\frac{16}{7}$&$53$&$212$&$0$&$1$&$(1,3,1,1)$&$245025$\\
\hline
\end{tabular}
\Table\label{t1} Eigenvalues of Laplacian  on $({\rm Spin}(9),\nu)$
\end{center}
\end{table}

\begin{table}[h]
\begin{center}
\begin{tabular}{|c|c|c|c|c|c|c|c|}
\hline
$\No$&$\lambda$&$k_1=21-14\lambda$&$k_2=84-56\lambda$&$L_4(k_1)$&$L_4(k_2)$&$(\nu_1,\nu_2,\nu_3,\nu_4)$&$\sigma(\lambda)$\\
\hline
$1$&$-\frac{4}{7}$&$29$&$116$&$0$&$1$&$(2,1,1,1)$&$81$\\
\hline
$2$&$-\frac{9}{7}$&$39$&$156$&$1$&$3$&$(3,1,1,1)$&$8992$\\
$ $&$ $&$ $&$ $&$ $&$ $&$(1,1,2,1)$&$ $\\
\hline
$3$&$-\frac{10}{7}$&$41$&$164$&$0$&$1$&$(1,1,1,3)$&$15876$\\
\hline
$4$&$-\frac{12}{7}$&$45$&$180$&$0$&$1$&$(2,2,1,1)$&$53361$\\
\hline
$5$&$-2$&$49$&$196$&$0$&$1$&$(2,1,2,1)$&$352836$\\
\hline
$6$&$-\frac{15}{7}$&$51$&$204$&$1$&$3$&$(4,1,1,1)$&$878112$\\
$ $&$ $&$ $&$ $&$ $&$ $&$(2,1,1,3)$&$ $\\
\hline
$7$&$-\frac{16}{7}$&$53$&$212$&$0$&$1$&$(1,3,1,1)$&$245025$\\
\hline
$8$&$-\frac{18}{7}$&$57$&$228$&$1$&$3$&$(3,2,1,1)$&$3550600$\\
$ $&$ $&$ $&$ $&$ $&$ $&$(1,2,2,1)$&$ $\\
\hline
$9$&$-\frac{19}{7}$&$59$&$236$&$0$&$1$&$(1,2,1,3)$&$7683984$\\
\hline
$10$&$-\frac{20}{7}$&$61$&$244$&$0$&$1$&$(3,1,2,1)$&$6036849$\\
\hline
\end{tabular}
\Table\label{t2} Eigenvalues of Laplacian  on $({\rm SO}(9),\nu)$
\end{center}
\end{table}

\begin{proof}
Following to (\ref{xi4}), let us consider Diophantine equation
\begin{equation}
\label{xx4}
(2\nu_1+2\nu_2+2\nu_3+\nu_4)^2+(2\nu_2+2\nu_3+\nu_4)^2+(2\nu_3+\nu_4)^2+\nu_4^2=84-56\lambda,
\end{equation}
where $\nu_1,\,\nu_2,\,\nu_3,\,\nu_4\in\mathbb{N}$.
It is clear that if the equation (\ref{xx4}) is solvable, then $-56\lambda\in\mathbb{N}$.

Assume at first that the equation (\ref{xx4}) has a solution  $(\nu_1,\nu_2,\nu_3,\nu_4)\in\mathbb{N}^4$ such that $\nu_4\equiv 1(\mbox{mod}\,2)$.
Let us introduce the following notation
$$x=\nu_4,\,\,y=2\nu_3+\nu_4,\,\,z=2\nu_2+2\nu_3+\nu_4,\,\,v=2\nu_1+2\nu_2+2\nu_3+\nu_4.$$
Then the equation (\ref{xx4}) will written down in the form
\begin{equation}
\label{xx5}
x^2+y^2+z^2+v^2=84-56\lambda,
\end{equation}
furthermore
\begin{equation}
\label{cont2}
x,\,y,\,z,\,v\in\mathbb{N},\,\,x\equiv y\equiv z\equiv v\equiv 1(\mbox{mod}\,2),\,\,x<y<z<v.
\end{equation}
Note that the square of any odd number gives residue $1$ under division by $8$, while the square of any even number
gives residue $0$ or $4$ under division by $8$. Consequently if the equation (\ref{xx5}) is solvable under condition (\ref{cont2}),
then $84-56\lambda\equiv 4(\mbox{mod}\,8)$ what is equivalent to relation $-7\lambda\in\mathbb{N}$. Conversely, if $-7\lambda\in\mathbb{N}$ 
and the equation (\ref{xx5}) is solvable under condition (\ref{cont2}), then or $x,\,y,\,z,\,v$ are odd, or $x,\,y,\,z,\,v$ are even
and  among these numbers the odd quantity is divisible by $4$.
Therefore the number of solutions to equation (\ref{xx5}) with condition (\ref{cont2}) is equal to $L_4(84-56\lambda)-L_4(21-14\lambda)$.

Suppose now that the equation (\ref{xx4}) has a solution $(\nu_1,\,\nu_2,\,\nu_3,\,\nu_4)\in\mathbb{N}^4$ such that $\nu_4\equiv 0(\mbox{mod}\,2)$.
It is clear that  $84-56\lambda\equiv 0(\mbox{mod}\,4)$ what is equivalent to relation $-14\lambda\in\mathbb{N}$. Let us introduce the following notation
$$x=\frac{\nu_4}{2},\,\,y=\nu_3+\frac{\nu_4}{2},\,\,z=\nu_2+\nu_3+\frac{\nu_4}{2},\,\,v=\nu_1+\nu_2+\nu_3+\frac{\nu_4}{2}.$$
Then the equation (\ref{xx4}) will written down in the form
\begin{equation}
\label{xx6}
x^2+y^2+z^2+v^2=21-14\lambda,
\end{equation}
furthermore
\begin{equation}
\label{cont3}
x,\,y,\,z,\,v\in\mathbb{N},\,\,x<y<z<v.
\end{equation}
In addition the number of solutions to equation (\ref{xx6}), satisfying conditions (\ref{cont3}), is equal to $L_4(21-14\lambda)$.
\end{proof}

\begin{theorem}
\label{x2}
Let $G={\rm SO}(9)$ is supplied by biinvariant Riemannian metric $\nu$ such that $\nu(e)=-k_{{\rm ad}}$.
 A number $\lambda < 0$ is an eigenvalue of Laplacian on $(G,\nu)$ if and only if the following conditions are fulfilled

1) $-7\lambda\in\mathbb{N}$;

2) natural number $84-56\lambda$ is a sum of squares of four mutually different odd natural numbers, i.e.
$L_4(21-14\lambda)<L_4(84-56\lambda)$.

In addition the number of highest weights $\Lambda$, such that $\lambda(\Lambda)=\lambda$, is equal to $L_4(84-56\lambda)-L_4(21-14\lambda)$.
\end{theorem}

\begin{proof}
This statement follows from (\ref{xi5}) and the proof of theorem \ref{x1}.
\end{proof}

\begin{table}[h]
\begin{center}
\begin{tabular}{|c|c|c|c|c|c|}
\hline
$\No$&$\lambda$&$30-20\lambda$&$L_4(30-20\lambda)$&$(\nu_1,\nu_2,\nu_3,\nu_4)$&$\sigma(\lambda)$\\
\hline
$1$&$-\frac{9}{20}$&$39$&$1$&$(2,1,1,1)$&$64$\\
\hline
$2$&$-\frac{4}{5}$&$46$&$1$&$(1,2,1,1)$&$729$\\
\hline
$3$&$-1$&$50$&$1$&$(3,1,1,1)$&$1296$\\
\hline
$4$&$-\frac{21}{20}$&$51$&$1$&$(1,1,2,1)$&$2304$\\
\hline
$5$&$-\frac{6}{5}$&$54$&$1$&$(1,1,1,2)$&$1764$\\
\hline
$6$&$-\frac{27}{20}$&$57$&$1$&$(2,2,1,1)$&$25600$\\
\hline
$7$&$-\frac{8}{5}$&$62$&$1$&$(2,1,2,1)$&$99225$\\
\hline
$8$&$-\frac{33}{20}$&$63$&$1$&$(4,1,1,1)$&$14400$\\
\hline
$9$&$-\frac{7}{4}$&$65$&$1$&$(2,1,1,2)$&$82944$\\
\hline
$10$&$-\frac{9}{5}$&$66$&$1$&$(1,3,1,1)$&$94864$\\
\hline
\end{tabular}
\Table\label{t3} Eigenvalues of Laplacian on $({\rm Sp}(4),\nu)$
\end{center}
\end{table}

\begin{table}[h]
\begin{center}
\begin{tabular}{|c|c|c|c|c|c|}
\hline
$\No$&$\lambda$&$30-20\lambda$&$L_4(30-20\lambda)$&$(\nu_1,\nu_2,\nu_3,\nu_4)$&$\sigma(\lambda)$\\
\hline
$1$&$-\frac{4}{5}$&$46$&$1$&$(1,2,1,1)$&$729$\\
\hline
$2$&$-1$&$50$&$1$&$(3,1,1,1)$&$1296$\\
\hline
$3$&$-\frac{6}{5}$&$54$&$1$&$(1,1,1,2)$&$1764$\\
\hline
$4$&$-\frac{8}{5}$&$62$&$1$&$(2,1,2,1)$&$99225$\\
\hline
$5$&$-\frac{9}{5}$&$66$&$1$&$(1,3,1,1)$&$94864$\\
\hline
$6$&$-2$&$70$&$1$&$(3,2,1,1)$&$352836$\\
\hline
$7$&$-\frac{11}{5}$&$74$&$1$&$(1,2,1,2)$&$627264$\\
\hline
$8$&$-\frac{12}{5}$&$78$&$3$&$(5,1,1,1)$&$2123550$\\
$ $&$ $&$ $&$ $&$(3,1,1,2)$&$ $\\
$ $&$ $&$ $&$ $&$(1,1,3,1)$&$ $\\
\hline
$9$&$-\frac{27}{10}$&$84$&$1$&$(2,2,2,1)$&$16777216$\\
\hline
$10$&$-\frac{14}{5}$&$86$&$1$&$(1,1,1,3)$&$352836$\\
\hline
\end{tabular}
\Table\label{t4} Eigenvalues of Laplacian on  $({\rm PSp}(4),\nu)$
\end{center}
\end{table}

\begin{theorem}
\label{x3}
Let $G={\rm SO}(9)$ is supplied by biinvariant Riemannian metric $\nu$ such that $\nu(e)=-k_{{\rm ad}}$.
 A number $\lambda < 0$ is an eigenvalue of Laplacian on $(G,\nu)$ if and only if the following conditions are fulfilled

1) $-20\lambda\in\mathbb{N}$;

2) natural number $30-20\lambda$ is a sum of squares of four mutually different natural numbers, i.e.
 $L_4(30-20\lambda)>0$.

In addition the number of highest weights $\Lambda$, such that $\lambda(\Lambda)=\lambda$, is equal to $L_4(30-20\lambda)$.
\end{theorem}

\begin{proof}
Following to (\ref{xi6}), let us consider Diophantine equation
\begin{equation}
\label{xx1}
(\nu_1+\nu_2+\nu_3+\nu_4)^2+(\nu_2+\nu_3+\nu_4)^2+(\nu_3+\nu_4)^2+\nu_4^2=30-20\lambda,
\end{equation}
where $\nu_1,\,\nu_2,\,\nu_3,\,\nu_4\in\mathbb{N}$.
It is clear that if the equation (\ref{xx1}) is solvable, then $-20\lambda\in\mathbb{N}$. Let us introduce the following notation
\begin{equation}
\label{not1}
x=\nu_4,\,\,y=\nu_3+\nu_4,\,\,z=\nu_2+\nu_3+\nu_4,\,\,v=\nu_1+\nu_2+\nu_3+\nu_4.
\end{equation}
Then the equation (\ref{xx1}) will written down in the form
\begin{equation}
\label{xx2}
x^2+y^2+z^2+v^2=30-20\lambda,
\end{equation}
moreover the condition (\ref{cont3}) is fulfilled.
Therefore the number of solutions to equation (\ref{xx1})  is equal to $L_4(30-20\lambda)$.
\end{proof}

\begin{theorem}
\label{x4}
Let $G={\rm PSp}(4)$ is supplied by biinvariant Riemannian metric $\nu$ such that $\nu(e)=-k_{{\rm ad}}$.
 A number $\lambda < 0$ is an eigenvalue of Laplacian on $(G,\nu)$ if and only if the following conditions are fulfilled

1) $-10\lambda\in\mathbb{N}$;

2) natural number  $30-20\lambda$ is a sum of squares of four mutually different natural numbers, i.e.
 $L_4(30-20\lambda)>0$.

In addition the number of highest weights $\Lambda$, such that $\lambda(\Lambda)=\lambda$, is equal to  $L_4(30-20\lambda)$.
\end{theorem}

\begin{proof}
Following to (\ref{xi7}), consider Diophantine equation (\ref{xx1}), where $\nu_1\equiv\nu_3({\rm mod}\,2)$, $\nu_1,\,\nu_2,\,\nu_3,\,\nu_4\in\mathbb{N}.$
 After the change of variables (\ref{not1}) the equation (\ref{xx1})  will written down in the form (\ref{xx2}), where
$$\,x,\,y,\,z,\,v\in\mathbb{N},\,\,x+v\equiv y+z({\rm mod}\,2),\,\,x<y<z<v.$$
If $x+v\equiv y+z({\rm mod}\,2)$ then $x^2+v^2\equiv y^2+z^2({\rm mod}\,2)$. Then it follows from (\ref{xx2}) that a number $k=30-20\lambda$ is an even
what is equivalent to relation  $-10\lambda\in\mathbb{N}$. Conversely, if $-10\lambda\in\mathbb{N}$ then $x+v\equiv y+z({\rm mod}\,2)$
on the ground of the equation (\ref{xx2}), where $x,\,y,\,z,\,v\in\mathbb{N}.$
 Therefore the number of solutions $(\nu_1,\nu_2,\nu_3,\nu_4)\in\mathbb{N}^4$ of equation (\ref{xx1}), satisfying the condition
 $\nu_1\equiv\nu_3({\rm mod}\,2)$, is equal to $L_4(30-20\lambda)$.
 \end{proof}

\begin{theorem}
\label{x5}
Let $G={\rm Spin}(8)$ is supplied by biinvariant Riemannian metric $\nu$ such that $\nu(e)=-k_{ad}$.
A number $\lambda < 0$ is an eigenvalue of Laplacian on $(G,\nu)$ in one and one of the following cases (I, II or III)

I. 1) $-12\lambda\equiv 1({\rm mod}\,2)$;

2) natural number   $56-48\lambda$ is a sum of squares of four mutually different odd natural numbers, i.e.
 $L_4(14-12\lambda)<L_4(56-48\lambda)$.

In addition the number of highest weights $\Lambda$, such that $\lambda(\Lambda)=\lambda$, is equal to  $2L_4(56-48\lambda)+L_3(14-12\lambda)$.

II. 1) $-12\lambda\in\mathbb{N}$;

2) natural number   $14-12\lambda$ is a sum of squares of four mutually different natural numbers, i.e. $L_4(14-12\lambda)>0$;

3) natural number   $56-48\lambda$ can't be presented as a sum of squares of four mutually different odd natural numbers,
i.e.  $L_4(56-48\lambda)=L_4(14-12\lambda)$.

In addition the number of highest weights $\Lambda$, such that $\lambda(\Lambda)=\lambda$, is equal to  $2L_4(14-12\lambda)+L_3(14-12\lambda)$.

III. 1)  $-12\lambda\in\mathbb{N}$;

2) natural number   $56-48\lambda$ can't be presented as a sum of squares of four mutually different natural numbers,
i.e. $L_4(56-48\lambda)=0$;

3) natural number   $14-12\lambda$ is a sum of squares of three mutually different natural numbers, i.e.
$L_3(14-12\lambda)>0$.

In addition the number of highest weights $\Lambda$, such that $\lambda(\Lambda)=\lambda$, is equal to  $L_3(14-12\lambda)$.
\end{theorem}

\begin{table}[h]
\begin{center}
\begin{tabular}{|c|c|c|c|c|c|c|c|c|}
\hline
$\No$&$\lambda$&$k_1=14-12\lambda$&$4k_1$&$L_4(k_1)$&$L_4(4k_1)$&$L_3(k_1)$&$(\nu_1,\nu_2,\nu_3,\nu_4)$&$\sigma(\lambda)$\\
\hline
$1$&$-\frac{7}{12}$&$21$&$84$&$0$&$1$&$1$&$(2,1,1,1)$&$192$\\
$ $&$ $&$ $&$ $&$ $&$ $&$ $&$(1,1,2,1)$&$ $\\
$ $&$ $&$ $&$ $&$ $&$ $&$ $&$(1,1,1,2)$&$ $\\
\hline
$2$&$-1$&$26$&$104$&$0$&$0$&$1$&$(1,2,1,1)$&$784$\\
\hline
$3$&$-\frac{5}{4}$&$29$&$116$&$0$&$1$&$1$&$(1,1,2,2)$&$9408$\\
$ $&$ $&$ $&$ $&$ $&$ $&$ $&$(2,1,2,1)$&$ $\\
$ $&$ $&$ $&$ $&$ $&$ $&$ $&$(2,1,1,2)$&$ $\\
\hline
$4$&$-\frac{4}{3}$&$30$&$120$&$1$&$1$&$1$&$(3,1,1,1)$&$3675$\\
$ $&$ $&$ $&$ $&$ $&$ $&$ $&$(1,1,3,1)$&$ $\\
$ $&$ $&$ $&$ $&$ $&$ $&$ $&$(1,1,1,3)$&$ $\\
\hline
$5$&$-\frac{7}{4}$&$35$&$140$&$0$&$1$&$1$&$(2,2,1,1)$&$76800$\\
$ $&$ $&$ $&$ $&$ $&$ $&$ $&$(1,2,1,2)$&$ $\\
$ $&$ $&$ $&$ $&$ $&$ $&$ $&$(1,2,2,1)$&$ $\\
\hline
\end{tabular}
\Table\label{t5} Eigenvalues of Laplacian on  $({\rm Spin}(8),\nu).$ I
\end{center}
\end{table}

\begin{table}[h]
\begin{center}
\begin{tabular}{|c|c|c|c|c|c|c|c|c|}
\hline
$\No$&$\lambda$&$k_1=14-12\lambda$&$4k_1$&$L_4(k_1)$&$L_4(4k_1)$&$L_3(k_1)$&$(\nu_1,\nu_2,\nu_3,\nu_4)$&$\sigma(\lambda)$\\
\hline
$6$&$-2$&$38$&$152$&$0$&$0$&$1$&$(2,1,2,2)$&$122500$\\
\hline
$7$&$-\frac{25}{12}$&$39$&$156$&$1$&$3$&$0$&$(3,1,2,1)$&$301056$\\
$ $&$ $&$ $&$ $&$ $&$ $&$ $&$(3,1,1,2)$&$ $\\
$ $&$ $&$ $&$ $&$ $&$ $&$ $&$(1,1,3,2)$&$ $\\
$ $&$ $&$ $&$ $&$ $&$ $&$ $&$(1,1,2,3)$&$ $\\
$ $&$ $&$ $&$ $&$ $&$ $&$ $&$(2,1,3,1)$&$ $\\
$ $&$ $&$ $&$ $&$ $&$ $&$ $&$(2,1,1,3)$&$ $\\
\hline
$8$&$-\frac{9}{4}$&$41$&$164$&$0$&$1$&$1$&$(4,1,1,1)$&$37632$\\
$ $&$ $&$ $&$ $&$ $&$ $&$ $&$(1,1,4,1)$&$ $\\
$ $&$ $&$ $&$ $&$ $&$ $&$ $&$(1,1,1,4)$&$ $\\
\hline
$9$&$-\frac{7}{3}$&$42$&$168$&$0$&$0$&$1$&$(1,3,1,1)$&$90000$\\
\hline
$10$&$-\frac{31}{12}$&$45$&$180$&$0$&$1$&$1$&$(1,2,2,2)$&$2116800$\\
$ $&$ $&$ $&$ $&$ $&$ $&$ $&$(2,2,2,1)$&$ $\\
$ $&$ $&$ $&$ $&$ $&$ $&$ $&$(2,2,1,2)$&$ $\\
\hline
\end{tabular}
\Table\label{t6} Eigenvalues of Laplacian on $({\rm Spin}(8),\nu)$. II
\end{center}
\end{table}

\begin{table}[h]
\begin{center}
\begin{tabular}{|c|c|c|c|c|c|c|c|c|}
\hline
$\No$&$\lambda$&$k_1=14-12\lambda$&$4k_1$&$L_4(k_1)$&$L_4(4k_1)$&$L_3(k_1)$&$(\nu_1,\nu_2,\nu_3,\nu_4)$&$\sigma(\lambda)$\\
\hline
$1$&$-\frac{7}{12}$&$21$&$84$&$0$&$1$&$1$&$(1,1,1,2)$&$64$\\
\hline
$2$&$-1$&$26$&$104$&$0$&$0$&$1$&$(1,2,1,1)$&$784$\\
\hline
$3$&$-\frac{5}{4}$&$29$&$116$&$0$&$1$&$1$&$(2,1,2,1)$&$3136$\\
\hline
$4$&$-\frac{4}{3}$&$30$&$120$&$1$&$1$&$1$&$(1,1,1,3)$&$3675$\\
$ $&$ $&$ $&$ $&$ $&$ $&$ $&$(1,1,3,1)$&$ $\\
$ $&$ $&$ $&$ $&$ $&$ $&$ $&$(3,1,1,1)$&$ $\\
\hline
$5$&$-\frac{7}{4}$&$35$&$140$&$0$&$1$&$1$&$(1,2,1,2)$&$25600$\\
\hline
$6$&$-2$&$38$&$152$&$0$&$0$&$1$&$(2,1,2,2)$&$122500$\\
\hline
$7$&$-\frac{25}{12}$&$39$&$156$&$1$&$3$&$0$&$(3,1,1,2)$&$100352$\\
$ $&$ $&$ $&$ $&$ $&$ $&$ $&$(1,3,1,2)$&$ $\\
\hline
$8$&$-\frac{9}{4}$&$41$&$164$&$0$&$1$&$1$&$(1,1,1,4)$&$12544$\\
\hline
$9$&$-\frac{7}{3}$&$42$&$168$&$0$&$0$&$1$&$(1,3,1,1)$&$90000$\\
\hline
$10$&$-\frac{31}{12}$&$45$&$180$&$0$&$1$&$1$&$(2,2,2,1)$&$705600$\\
\hline
\end{tabular}
\Table\label{t7} Eigenvalues of Laplacian on  $({\rm SO}(8),\nu)$
\end{center}
\end{table}

\begin{proof}
Following to (\ref{xi8}), let us consider Diophantine equation
\begin{equation}
\label{xx7}
(2\nu_1+2\nu_2+\nu_3+\nu_4)^2+(2\nu_2+\nu_3+\nu_4)^2+(\nu_3+\nu_4)^2+(\nu_3-\nu_4)^2=56-48\lambda,
\end{equation}
where $\nu_1,\,\nu_2,\,\nu_3,\,\nu_4\in\mathbb{N}.$
It is clear that if the equation (\ref{xx7}) is solvable, then $-48\lambda\in\mathbb{N}$.

Assume at first that the equation (\ref{xx7}) has a solution $(\nu_1,\nu_2,\nu_3,\nu_4)\in\mathbb{N}^4$ such that $\nu_3+\nu_4\equiv 1({\rm mod}\,2)$.
Since  $\nu_3$ and $\nu_4$ occur symmetrically in (\ref{xx7}),  then
we can suppose for definiteness that $\nu_3>\nu_4$. Let us introduce the following notation
\begin{equation}
\label{sos1}
x=\nu_3-\nu_4,\,\,y=\nu_3+\nu_4,\,\,z=2\nu_2+\nu_3+\nu_4,\,\,v=2\nu_1+2\nu_2+\nu_3+\nu_4.
\end{equation}
Then the equation (\ref{xx7}) will written down in the form
\begin{equation}
\label{xx8}
x^2+y^2+z^2+v^2=56-48\lambda,
\end{equation}
where (\ref{cont2}).  It follows from the proof of theorem \ref{x1} that if the equation (\ref{xx8})  is solvable under condition
(\ref{cont2}), then $56-48\lambda\equiv 4({\rm mod}\,8)$ what is equivalent to relation   $-12\lambda\equiv 1({\rm mod}\,2)$.
Therefore the number of solutions $(\nu_1,\nu_2,\nu_3,\nu_4)\in\mathbb{N}^4$ of equation (\ref{xx7}), satisfying the condition
$\nu_3+\nu_4\equiv 1({\rm mod}\,2)$, is equal to $2L_4(56-48\lambda)-2L_4(14-12\lambda)$.

Let us suppose now that the equation (\ref{xx7}) has a solution $(\nu_1,\nu_2,\nu_3,\nu_4)$, where $\nu_3=\nu_4$. Then $-12\lambda\in\mathbb{N}$.
After the change of variables
\begin{equation}
\label{sos2}
x=\nu_3,\,\,y=\nu_2+\nu_3,\,\,z=\nu_1+\nu_2+\nu_3
\end{equation}
the equation (\ref{xx7}) will written down in the form
\begin{equation}
\label{qaq}
x^2+y^2+z^2=14-12\lambda,
\end{equation}
in addition
$x,\,y,\,z\in\mathbb{N}$, $x<y<z.$ Therefore the number of solutions $(\nu_1,\nu_2,\nu_3,\nu_4)\in\mathbb{N}^4$ to equation (\ref{xx7}),
satisfying condition $\nu_3=\nu_4$, is equal to $L_3(14-12\lambda)$.

Let suppose at the end that the equation  (\ref{xx7}) has a solution  $(\nu_1,\nu_2,\nu_3,\nu_4)\in\mathbb{N}^4$ such that $\nu_3\neq\nu_4$ and
$\nu_3\equiv\nu_4({\rm mod}\,2)$.
Then $-12\lambda\in\mathbb{N}$. Let us set for definiteness that  $\nu_3>\nu_4$ and introduce the following notation
\begin{equation}
\label{sos3}
x=\frac{\nu_3-\nu_4}{2},\,\,y=\frac{\nu_3+\nu_4}{2},\,\,z=\nu_2+\frac{\nu_3+\nu_4}{2},\,\,v=\nu_1+\nu_2+\frac{\nu_3+\nu_4}{2}.
\end{equation}
Then the equation (\ref{xx7}) will written down in the form
\begin{equation}
\label{qaq2}
x^2+y^2+z^2+v^2=14-12\lambda,
\end{equation}
where (\ref{cont3}).
Therefore the number of solutions $(\nu_1,\nu_2,\nu_3,\nu_4)\in\mathbb{N}^4$ of equation (\ref{xx7}), satisfying the condition $\nu_3\neq\nu_4$,
$\nu_3\equiv\nu_4({\rm mod}\,2)$, is equal to $2L_4(14-12\lambda)$.
\end{proof}

\begin{theorem}
\label{x04}
Let $G={\rm SO}(8)$ is supplied by biinvariant metric $\nu$ such that $\nu(e)=-k_{ad}$. A number
$\lambda < 0$ is an eigenvalue of Laplacian on $(G,\nu)$ for one and only one of the following cases

I. 1) $-12\lambda\equiv 1({\rm mod}\,2)$;

2) natural number  $56-48\lambda$ is a sum of squares of four mutually odd different numbers, i.e.
$L_4(56-48\lambda)>L_4(14-12\lambda)$.

In addition the number of highest weights $\Lambda,$ such that $\lambda(\Lambda)=\lambda$, is equal to  $L_4(56-48\lambda)-L_4(14-12\lambda)$.

II. 1) $-6\lambda\in\mathbb{N}$;

2) natural number  $14-12\lambda$ is a sum of squares of four mutually different numbers, i.e. $L_4(14-12\lambda)>0$;

3) natural number  $56-48\lambda$ can't be presented as a sum of squares of four mutually different odd numbers, i.e.
$L_4(56-48\lambda)=L_4(14-12\lambda)$.

In addition the number of highest weights $\Lambda,$ such that $\lambda(\Lambda)=\lambda$, is equal to  $2L_4(14-12\lambda)+L_3(14-12\lambda)$.

III. 1) $-6\lambda\in\mathbb{N}$;

2) natural number $56-48\lambda$ can't be presented as a sum of squares of four mutually different numbers, i.e.  $L_4(56-48\lambda)=0$;

3) natural number $14-12\lambda$ is a sum of squares of three mutually different numbers, i.e. $L_3(14-12\lambda)>0$.

In addition the number of highest weights $\Lambda,$ such that $\lambda(\Lambda)=\lambda$, is equal to  $L_3(14-12\lambda)$.

\end{theorem}

\begin{table}[h]
\begin{center}
\begin{tabular}{|c|c|c|c|c|c|c|c|}
\hline
$\No$&$\lambda$&$14-12\lambda$&$L_4(14-12\lambda)$&$L_3(14-12\lambda)$&$(\nu_1,\nu_2,\nu_3,\nu_4)$&$\sigma(\lambda)$\\
\hline
$1$&$-1$&$26$&$0$&$1$&$(1,2,1,1)$&$784$\\
\hline
$2$&$-\frac{4}{3}$&$30$&$1$&$1$&$(1,1,1,3)$&$3675$\\
$ $&$ $&$ $&$ $&$ $&$(1,1,3,1)$&$ $\\
$ $&$ $&$ $&$ $&$ $&$(3,1,1,1)$&$ $\\
\hline
$3$&$-2$&$38$&$0$&$1$&$(2,1,2,2)$&$122500$\\
\hline
$4$&$-\frac{7}{3}$&$42$&$0$&$1$&$(1,3,1,1)$&$90000$\\
\hline
$5$&$-\frac{8}{3}$&$46$&$1$&$1$&$(3,2,1,1)$&$964467$\\
$ $&$ $&$ $&$ $&$ $&$(1,2,3,1)$&$ $\\
$ $&$ $&$ $&$ $&$ $&$(1,2,1,3)$&$ $\\
\hline
$6$&$-3$&$50$&$1$&$1$&$(1,1,3,3)$&$2116800$\\
$ $&$ $&$ $&$ $&$ $&$(3,1,3,1)$&$ $\\
$ $&$ $&$ $&$ $&$ $&$(3,1,1,3)$&$ $\\
\hline
$7$&$-\frac{10}{3}$&$54$&$1$&$1$&$(5,1,1,1)$&$259308$\\
$ $&$ $&$ $&$ $&$ $&$(1,1,5,1)$&$ $\\
$ $&$ $&$ $&$ $&$ $&$(1,1,1,5)$&$ $\\
\hline
$8$&$-\frac{7}{2}$&$56$&$0$&$1$&$(2,2,2,2)$&$16777216$\\
\hline
$9$&$-4$&$62$&$1$&$2$&$(1,4,1,1)$&$44222500$\\
$ $&$ $&$ $&$ $&$ $&$(4,1,2,2)$&$ $\\
$ $&$ $&$ $&$ $&$ $&$(2,1,4,2)$&$ $\\
$ $&$ $&$ $&$ $&$ $&$(2,1,2,4)$&$ $\\
\hline
$10$&$-\frac{13}{3}$&$66$&$1$&$1$&$(3,3,1,1)$&$55780032$\\
$ $&$ $&$ $&$ $&$ $&$(1,3,3,1)$&$ $\\
$ $&$ $&$ $&$ $&$ $&$(1,3,1,3)$&$ $\\
\hline
\end{tabular}
\Table\label{t8} Eigenvalues of Laplacian on   $({\rm PSO}(8),\nu)$
\end{center}
\end{table}

\begin{proof}
Following to (\ref{xi9}), let us consider Diophantine equation (\ref{xx7}),
where $\nu_1\equiv\nu_3({\rm mod}\,2)$, $\nu_1,\,\nu_2,\,\nu_3,\,\nu_4\in\mathbb{N}$.
It is clear that if the equation (\ref{xx7}) is solvable then $-48\lambda\in\mathbb{N}$.

Assume at first that the equation (\ref{xx7}) has a solution $(\nu_1,\nu_2,\nu_3,\nu_4)\in\mathbb{N}^4$ such that $\nu_1\equiv\nu_3(\mbox{mod}\,2)$ and
$\nu_3+\nu_4\equiv 1(\mbox{mod}\,2)$. We can suppose without loss of generality that $\nu_3>\nu_4$.
After the change of variables (\ref{sos1}) the equation (\ref{xx7}) will written down in the form  (\ref{xx8}),
moreover $x,\,y,\,z,\,v\in\mathbb{N}$,
\begin{equation}
\label{as}
x\equiv y\equiv z\equiv v\equiv 1({\rm mod}\,2),\,\,x<y<z<v,\,\,x+y+z+v\equiv 2({\rm mod}\,4).
\end{equation}
Let $x$, $y$, $z$, $v$ are odd. In this case it is easy to prove that the condition $x+y+z+v\equiv 2(\mbox{mod}\,4)$ is equivalent to condition
 $x^2+y^2+z^2+v^2\equiv 4(\mbox{mod}\,16)$ or $x^2+y^2+z^2+v^2\equiv 12(\mbox{mod}\,16)$. Then on the ground of (\ref{xx8}) we get
 $56-48\lambda\equiv 4(\mbox{mod}\,16)$ what is equivalent to relation
$-12\lambda\equiv 3(\mbox{mod}\,4)$, or $56-48\lambda\equiv 12(\mbox{mod}\,16)$ what is equivalent
to relation $-12\lambda\equiv 1(\mbox{mod}\,4)$. Thus if  $x$, $y$, $z$, $v$ are odd,  then the condition $x+y+z+v\equiv 2(\mbox{mod}\,4)$ is equivalent to condition
 $-12\lambda\equiv 1(\mbox{mod}\,2)$. In addition the number of solutions $(\nu_1,\nu_2,\nu_3,\nu_4)\in\mathbb{N}^4$  to equation (\ref{xx7}), 
 satisfying  conditions
 $\nu_1\equiv\nu_3(\mbox{mod}\,2)$ and $\nu_3+\nu_4\equiv 1(\mbox{mod}\,2)$, is equal to $L_4(56-48\lambda)-L_4(14-12\lambda)$.

Suppose now that the equation (\ref{xx7}) has a solution $(\nu_1,\nu_2,\nu_3,\nu_4)\in\mathbb{N}^4$ such that $\nu_1\equiv\nu_3(\mbox{mod}\,2)$ and
$\nu_3=\nu_4$. After the change of variables
(\ref{sos2}) the equation (\ref{xx7}) can be written down in the form (\ref{qaq}), furthermore
$$x,\,y,\,z\in\mathbb{N},\,\,x<y<z,\,\,z\equiv x+y(\mbox{mod}\,2).$$
If $z\equiv x+y(\mbox{mod}\,2)$ then $z^2\equiv x^2+y^2(\mbox{mod}\,2)$. Then $14-12\lambda$ is an even number what is
equivalent to relation  $-6\lambda\in\mathbb{N}$.
Conversely, if $-6\lambda\in\mathbb{N}$ then $z\equiv x+y(\mbox{mod}\,2)$.
In addition the number of solutions $(\nu_1,\nu_2,\nu_3,\nu_4)\in\mathbb{N}^4$ to equation (\ref{xx7}), satisfying conditions
$\nu_1\equiv\nu_3(\mbox{mod}\,2)$ and  $\nu_3=\nu_4$, is equal to $L_3(14-12\lambda)$.

Suppose now that the equation (\ref{xx7}) has a solution $(\nu_1,\nu_2,\nu_3,\nu_4)\in\mathbb{N}^4$ such that
$\nu_1\equiv\nu_3\equiv\nu_4(\mbox{mod}\,2)$ and $\nu_3\neq\nu_4$. Then $-12\lambda\in\mathbb{N}$.
We can suppose without loss of generality that $\nu_3>\nu_4$.  After the change of variables (\ref{sos3}) the equation
 (\ref{xx7}) will written down in the form (\ref{qaq2}), moreover
$$x,\,y,\,z,\,v\in\mathbb{N},\,\,x<y<z<v,\,\,v\equiv x+y+z(\mbox{mod}\,2).$$
If $v\equiv x+y+z(\mbox{mod}\,2)$ then $v^2\equiv x^2+y^2+z^2(\mbox{mod}\,2)$. Then it follows from (\ref{qaq2}) that a natural number 
 $14-12\lambda$ is even what is equivalent  to relation $-6\lambda\in\mathbb{N}$. Conversely, if $-6\lambda\in\mathbb{N}$ then we get 
$v\equiv x+y+z(\mbox{mod}\,2)$ in the ground of the equation (\ref{qaq2}), where $x,\,y,\,z\in\mathbb{N}$.
Therefore the number of solutions $(\nu_1,\nu_2,\nu_3,\nu_4)\in\mathbb{N}^4$ to equation (\ref{xx7}), satisfying
the condition $\nu_1\equiv\nu_3\equiv\nu_4(\mbox{mod}\,2)$ and $\nu_3\neq\nu_4$, is equal to $2L_4(14-12\lambda)$.
\end{proof}

\begin{theorem}
\label{x05}
Let $G={\rm PSO}(8)$ is supplied by biinvariant metric $\nu$ such that $\nu(e)=-k_{ad}$. A number
$\lambda < 0$ is an eigenvalue of Laplacian on $(G,\nu)$ for one and only one of the following cases (I or II)

I. 1) $-6\lambda\in\mathbb{N}$;

2) natural number  $14-12\lambda$ is a sum of squares of four mutually different numbers, i.e. $L_4(14-12\lambda)>0$.

In addition the number of highest weights $\Lambda,$ such that $\lambda(\Lambda)=\lambda$, is equal to  $2L_4(14-12\lambda)+L_3(14-12\lambda)$.

II. 1) $-6\lambda\in\mathbb{N}$;

2) natural number  $14-12\lambda$ can't be presented as a sum of squares of four mutually different numbers, i.e. $L_4(14-12\lambda)=0$.

3) natural number  $14-12\lambda$ is a sum of squares of three mutually different numbers, i.e. $L_3(14-12\lambda)>0$.

In addition the number of highest weights $\Lambda,$ such that $\lambda(\Lambda)=\lambda$, is equal to $L_3(14-12\lambda)$.
\end{theorem}

\begin{proof}
It follows from (\ref{xi10}) and proof of theorem \ref{x04}.
\end{proof}

\end{document}